\documentclass[11pt]{amsart}
\usepackage{amssymb,latexsym,amsmath,graphicx,graphics,epic,eepic}

\theoremstyle{plain}
\newtheorem{theorem}{Theorem}
\numberwithin{theorem}{section}
\newtheorem{lemma}[theorem]{Lemma}
\newtheorem{proposition}[theorem]{Proposition}
\newtheorem{corollary}[theorem]{Corollary}

\theoremstyle{definition}
\newtheorem{definition}[theorem]{Definition}

\newtheorem{question}[theorem]{Question}
\newtheorem{remark}[theorem]{Remark}

\newcommand{\C}{{\mathbb C}}
\newcommand{\R}{{\mathbb R}}
\newcommand{\Z}{{\mathbb Z}}

\renewcommand{\P}{{\mathbb P}}
\newcommand{\s}{{\mathbb S}}

\begin{document}
\title{Fixed-point free circle actions on $4$-manifolds}
\author{Weimin Chen}
\subjclass[2000]{Primary 57S15, Secondary 57M07, 57M50}
\keywords{}
\thanks{The author is partially supported by NSF grant DMS-1065784.}
\date{\today}
\maketitle

\begin{abstract}
This paper is concerned with fixed-point free $\s^1$-actions (smooth or locally linear) on
orientable $4$-manifolds. We show that the fundamental group plays a predominant role in the
equivariant classification of such $4$-manifolds. In particular, it is shown that for any finitely 
presented group with infinite center, there are at most finitely many distinct smooth (resp.
topological) $4$-manifolds which support a fixed-point free smooth (resp. locally linear)
$\s^1$-action and realize the given group as the fundamental group. A similar statement 
holds for the number of equivalence classes of fixed-point free $\s^1$-actions under some 
further conditions on the fundamental group. The connection between the classification 
of the $\s^1$-manifolds and the fundamental group is given by a certain decomposition, 
called {\it fiber-sum decomposition}, of the $\s^1$-manifolds. More concretely, each fiber-sum 
decomposition naturally gives rise to a Z-splitting of the fundamental group. There are two 
technical results in this paper which play a central role in our considerations. One states that 
the Z-splitting is a canonical JSJ decomposition of the fundamental group in the sense of 
Rips and Sela \cite{RipS}. Another asserts that if the fundamental group has infinite center, 
then the homotopy class of principal orbits of any fixed-point free $\s^1$-action on the 
$4$-manifold must be infinite, unless the $4$-manifold is the mapping torus of a periodic 
diffeomorphism of some elliptic $3$-manifold.
\end{abstract}

\section{Introduction}

Locally linear $\s^1$-actions on oriented $4$-manifolds were classified by Fintushel up to 
orientation-preserving equivariant homeomorphisms (for smooth $\s^1$-actions the classification 
is up to orientation-preserving equivariant diffeomorphisms), cf. \cite{F0,F1,F2}. 
One associates to each locally linear $\s^1$-action a legally weighted $3$-manifold, which is the 
orbit space decorated with certain orbit-type data and a characteristic class of the $\s^1$-action. 
The equivariant classification of the $\s^1$-four-manifolds is then given by the isomorphism classes 
of the corresponding legally weighted $3$-manifolds. 

An important technique for studying locally linear $\s^1$-actions on $4$-manifolds is a replacement
trick due to Pao \cite{Pao}. Pao's trick allows one to trade a certain weighted circle in a legally weighted
$3$-manifold for a pair of fixed points, or to have the weighted circle deleted and a $3$-ball removed
from the legally weighted $3$-manifold. (In particular, Pao's replacement trick applies only to locally 
linear $\s^1$-actions with a nonempty fixed-point set.) This procedure has the effect of replacing the 
given $\s^1$-action by another (non-equivalent) $\s^1$-action on the same $4$-manifold. Besides the 
construction of locally linear, nonlinear $\s^1$-actions on $\s^4$ in the original paper \cite{Pao}, the 
following are some of the further implications of Pao's trick when combined with the classification 
results of Fintushel in \cite{F0,F1,F2}:

\begin{itemize}
\item If a $4$-manifold $X$ admits a locally linear (resp. smooth) $\s^1$-action with a pair of fixed 
points or a fixed $2$-sphere, then $X$ admits infinitely many non-equivalent locally linear (resp. smooth) 
$\s^1$-actions (cf. \cite{Pao}). (There are many examples of such $4$-manifolds, including a large 
class of simply-connected $4$-manifolds.)
\item Modulo the $3$-dimensional Poincar\'{e} conjecture (which is now resolved \cite{P}), 
a simply-connected, smooth $\s^1$-four-manifold is diffeomorphic to a connected sum of $\s^4$,
$\pm \C\P^2$, or $\s^2\times\s^2$ (cf. \cite{F2}, compare also \cite{Y}).
\item If an oriented $4$-manifold with $b_2^{+}\geq 1$ admits a locally linear (resp. smooth) $\s^1$-action having at least one fixed point, then it contains a topologically (resp. smoothly) embedded, essential,
$2$-sphere of non-negative self-intersection (cf. Baldridge \cite{Bald1}, Theorem 2.1).\footnote{Baldridge \cite{Bald1} works in the smooth category, but the arguments are valid in the locally linear category as well.} In particular, the Hurwitz map $\pi_2\rightarrow H_2$ has infinite image. Baldridge's theorem gives a useful obstruction for the existence of $\s^1$-actions with fixed points, particularly for the smooth case as such a smoothly embedded $2$-sphere constrains the Seiberg-Witten invariants of the $4$-manifold (cf. \cite{FS}). 
\end{itemize}

In this paper we study fixed-point free $\s^1$-actions on orientable $4$-manifolds, either smooth or 
locally linear, depending on which category (i.e., smooth or topological) we work in. The arguments 
are valid for both categories; for simplicity, we shall work mainly in the smooth category. 
Our results indicate that the equivariant classification of fixed-point free $\s^1$-actions, where there is
a lack of Pao's replacement trick, is sharply different from that of $\s^1$-actions with fixed points. 
In particular, we show that under reasonable assumptions, the fundamental group plays a predominant 
role in the equivariant classification of $4$-manifolds with a fixed-point free $\s^1$-action. We 
showcase this phenomena with the following two theorems. 

\begin{theorem}
Let $X$ be an orientable $4$-manifold such that (i) the center of $\pi_1(X)$ is infinite cyclic,
(ii) $\pi_1(X)$ is single-ended and is not isomorphic to the fundamental group of a Klein bottle, 
(iii) any canonical JSJ decomposition of $\pi_1(X)$ contains a vertex 
subgroup which is not isomorphic to an HNN extension of a finite cyclic group.
Then there exists a constant $C>0$ depending only on $\pi_1(X)$, such that the number of
equivalence classes of fixed-point free $\s^1$-actions (smooth or locally linear) on $X$ is bounded by $C$. 
\end{theorem}

\begin{theorem}
Let $G$ be a finitely presented group with infinite center. There exists a constant 
$C>0$ depending only on $G$, such that the number of diffeomorphism classes 
(resp. homeomorphism classes) of orientable $4$-manifolds admitting a fixed-point free, 
smooth (resp. locally linear) $\s^1$-action, whose fundamental group is isomorphic to $G$, 
is bounded by $C$. 
\end{theorem}

Our approach to equivariant classification (resp. classification) of fixed-point free $\s^1$-four-manifolds
differs from the traditional approach of legally weighted $3$-manifolds (cf. Fintushel \cite{F0,F1,F2}),
where in our method geometric group theory played a prominent role. The central notion in our 
approach is a certain decomposition of the $\s^1$-manifolds which are called a {\it fiber-sum 
decomposition} (see Definition 1.3). Each such a decomposition gives rise to a Z-splitting of the 
fundamental group of the manifold, and the central result of this paper states that the Z-splitting is a 
canonical JSJ decomposition of the fundamental group in the sense of Rips and Sela \cite{RipS} 
(see Theorem 1.5). We shall also point out that the methods of this paper are also essentially different
from those in Hillman \cite{Hillman}, where homotopy/homeomorphism classifications of $\s^1$-bundles
over certain $3$-manifolds are given. In particular, the diffeomorphism classification result in 
Theorem 1.2 is not accessible by the surgery theoretic techniques employed in Hillman \cite{Hillman}. 

The orbit map of a fixed-point free $\s^1$-action on an orientable $4$-manifold defines 
a Seifert-type $\s^1$-fibration of the $4$-manifold, giving the orbit space a structure of a closed, 
orientable $3$-dimensional orbifold whose singular set consists of a disjoint union of embedded circles,
called {\it singular circles}. (Equivalently, the $4$-manifold is the total space of a principal $\s^1$-bundle 
over the $3$-orbifold.) With this understood, 
the building blocks of a fiber-sum decomposition are oriented fixed-point free 
$\s^1$-four-manifolds whose corresponding orbit space is an irreducible $3$-orbifold. We shall call 
such $\s^1$-four-manifolds {\it irreducible}. Note that the orientation of the $4$-manifold determines an orientation of the base $3$-orbifold, as the fibers of the Seifert-type $\s^1$-fibration are canonically 
oriented. 

\begin{definition}(Fiber-sum decomposition)
Let $X$ be a smooth orientable $4$-manifold. 
Suppose we are given with a finite set of smooth {\it oriented} $4$-manifolds $X_i$, $i\in I$, 
with the following significance.
\begin{itemize}
\item [{(i)}] For each $i\in I$, there is a fixed-point free $\s^1$-action on $X_i$ with orbit map
$\pi_i:X_i\rightarrow Y_i$ where $Y_i$ is irreducible. 
\item [{(ii)}] There is a finite set $J$, such that for each $j\in J$, there exists a pair of distinct points
$y_{j,1}$, $y_{j,2}\in \sqcup_{i\in I} Y_i$, which have the same multiplicity if singular.
\item [{(iii)}] Let $F_{j,1},F_{j,2}$ be the fibers of 
$\sqcup_i\pi_i:\sqcup_i X_i\rightarrow \sqcup_i Y_i$ over
$y_{j,1}$, $y_{j,2}$ respectively. For each $j\in J$, there is an orientation-reversing but fiber-wise 
orientation-preserving, fiber-preserving diffeomorphism 
$\phi_j: \partial Nd(F_{j,1})\rightarrow \partial Nd(F_{j,2})$.
\item [{(iv)}] For any $i\in I$, $j\in J$, if $Y_i$ contains exactly one of the points $y_{j,1}$,
$y_{j,2}$, say $y_{j,1}\in Y_i$, then the homotopy class of the fiber $F_{j,1}$ generates a
{\it proper} subgroup of $\pi_1(X_i)$.
\end{itemize}
With the above understood, we say that $X$ admits a fiber-sum decomposition if there exists a
diffeomorphism between $X$ and the oriented $4$-manifold 
$$
\sqcup_{i\in I} X_i\setminus \sqcup_{j\in J} (Nd(F_{j,1})\sqcup Nd(F_{j,2}))/\sim \sqcup_{j\in J}\phi_j,
$$
and given such a diffeomorphism, we say that $X$ is {\it fiber-sum-decomposed} into $X_i$ along
$N_j$, where each $N_j\cong \s^1\times\s^2$ is the image of $\partial Nd(F_{j,1})$ (or equivalently,
$\partial Nd(F_{j,2})$) in $X$. Furthermore, the irreducible $\s^1$-four-manifolds $X_i$
are called the {\it factors} of the fiber-sum decomposition.
\end{definition}

\noindent{\bf Remarks: } The isotopy classification of diffeomorphisms of $\s^1\times \s^2$ is given
by $\pi_0(O(2)\times O(3)\times \Omega O(3))$, cf. Hatcher \cite{Hat}. In particular, there are two
distinct isotopy classes of homologically trivial diffeomorphisms because of the factor 
$\pi_0(\Omega O(3))=\pi_1 SO(3)=\Z_2$. However, the isotopy class of
the diffeomorphism $\phi_j:\partial Nd(F_{j,1})\rightarrow \partial Nd(F_{j,2})$ is uniquely determined 
because of the requirement of fiber-preserving.

\vspace{2mm}

It turns out that the class of $4$-manifolds which admit a fiber-sum decomposition are precisely 
the smooth, fixed-point free $\s^1$-four-manifolds whose fundamental group has infinite center. 
In order to understand this, we recall that a fixed-point free $\s^1$-action is called {\it injective} 
(and so is the corresponding $\s^1$-four-manifold), if the homotopy class of the principal orbits 
has infinite order. With this understood, note that in Definition 1.3, each $3$-orbifold $Y_i$ is 
irreducible. It follows easily that the $\s^1$-action on each $X_i$ must be injective. Moreover, 
it is clear that the $\s^1$-actions on $X_i$ descend to a fixed-point free $\s^1$-action 
on $X$, which is also injective. On the other hand, given any injective $\s^1$-action, the orbit
space (as a $3$-orbifold) admits a certain kind of spherical decompositions which are called 
{\it reduced} (see Lemma 2.4 for details), and any such a spherical decomposition naturally 
gives rise to a fiber-sum decomposition of the $4$-manifold (for more details see the proof of
Theorem 1.4). 

In summary, a $4$-manifold admits a fiber-sum decomposition if and only if it admits an injective 
fixed-point free $\s^1$-action. Note that the homotopy class of the principal orbits of the $\s^1$-action
lies in the center of the fundamental group of the $4$-manifold. In particular, the $\pi_1$ of an injective 
fixed-point free $\s^1$-four-manifold has infinite center. The converse is given in the following theorem. 

\begin{theorem}
Let $X$ be a smooth (resp. locally linear), fixed-point free $\s^1$-four-manifold whose 
fundamental group has infinite center. Then the $\s^1$-action must be injective unless 
$X$ is diffeomorphic (resp. homeomorphic) to the mapping torus of a periodic diffeomorphism 
of some elliptic $3$-manifold. As a consequence, any smooth, fixed-point free $\s^1$-four-manifold 
whose $\pi_1$ has infinite center admits a fiber-sum decomposition. 
\end{theorem}

Note that in the case where $X$ is diffeomorphic to the mapping torus of a periodic diffeomorphism 
of some elliptic $3$-manifold, $X$ admits another fixed-point free $\s^1$-action which is injective.
So in any event, the $4$-manifold admits a fiber-sum decomposition. 

We remark that the fundamental group of a smooth, fixed-point free $\s^1$-four-manifold with 
nontrivial Seiberg-Witten invariant must have infinite center, cf. \cite{C1,C2}. 

\vspace{3mm}

With the preceding understood, the main theme of this paper is to recover the fiber-sum 
decompositions of an injective $\s^1$-four-manifold from its fundamental group. The main results 
are summarized in Theorems 1.5 and 1.6 below. 

In order to describe the results, observe
that given any fiber-sum decomposition of $X$ into factors $X_i$ along $N_j$, there is an
associated finite graph of groups where the vertex groups and edge groups are given by
$\pi_1(X_i)$ and $\pi_1(N_j)$ respectively, such that $\pi_1(X)$ is isomorphic to the 
fundamental group of the graph of groups. Such a presentation of $\pi_1(X)$ is
called a {\it Z-splitting} as each edge group $\pi_1(N_j)$ is infinite cyclic. An in-depth study of 
Z-splittings of single-ended finitely generated groups was given in \cite{RipS} by Rips and Sela; 
in particular, they showed the existence of certain ``universal" Z-splittings for each single-ended 
finitely presented group, which are called {\it canonical JSJ decompositions}.

Theorem 1.5, which is the main technical result of this paper, asserts that the Z-splitting
associated to a fiber-sum-decomposition is a canonical JSJ decomposition of the fundamental 
group. 

\begin{theorem}
Let $X$ (resp. $X^\prime$) be a smooth $4$-manifold which is fiber-sum-decomposed into 
$X_i$ along $N_j$ (resp. $X^\prime_{i^\prime}$ along $N^\prime_{j^\prime}$). 
Suppose $\pi_1(X)$ (resp. $\pi_1(X^\prime)$) is single-ended and is not isomorphic to the 
fundamental group of a $2$-torus or a Klein bottle. Then the following hold.
\begin{itemize}
\item [{(1)}] The Z-splitting of $\pi_1(X)$ associated to the given fiber-sum-decomposition of 
$X$ is a canonical JSJ decomposition.\footnote{There is an annoying collapse of terminology here
as a canonical JSJ decomposition of $\pi_1(X)$ corresponds {\it not} to the JSJ decomposition of the
base $3$-orbifold, but to a reduced spherical decomposition of the $3$-orbifold, cf. Lemma 2.4.}
\item [{(2)}] Assume further that the submanifolds $N_j$, $N^\prime_{j^\prime}$ are 
null-homologous in $X$, $X^\prime$ respectively, and let 
$\alpha: \pi_1(X)\rightarrow \pi_1(X^\prime)$ 
be any isomorphism. Then after modifying the embeddings of $N_j$, $N^\prime_{j^\prime}$ by
fiber-preserving isotopies if necessary, $\alpha:  \pi_1(X)\rightarrow \pi_1(X^\prime)$ may be 
enhanced to an isomorphism between the Z-splittings of $\pi_1(X)$ and $\pi_1(X^\prime)$ 
associated to the new fiber-sum decompositions of $X$ and $X^\prime$ respectively. 
\end{itemize}
\end{theorem}

\noindent{\bf Remarks}
(1) Canonical JSJ decompositions are not unique as Z-splittings. Nevertheless, 
Theorem 1.5(1) implies that the number of factors $X_i$,  the number of submanifolds $N_j$, 
as well as the conjugacy classes of subgroups $\pi_1(X_i)$ and $\pi_1(N_j)$, depend only on 
$\pi_1(X)$ (see Proposition 3.5 for details). We shall also point out that in the course of
the proof of Theorem 1.5, the group $\pi_1(X)$ is shown to have the property that it admits 
no hyperbolic-hyperbolic elementary Z-splittings (cf. Lemma 3.1). 

(2) The stronger uniqueness in Theorem 1.5(2) corresponds to the uniqueness of canonical 
JSJ decompositions up to a sequence of {\it slidings, conjugations, and conjugations of boundary monomorphisms}. Such uniqueness has been established for torsion-free (Gromov) hyperbolic groups 
(cf. Sela \cite{Sela}, Theorem 1.7), but remains open in general for single-ended finitely presented 
groups (see \cite{RipS}, page 106).

(3) The assumption that the submanifolds $N_j$ are null-homologous in $X$ is equivalent to 
that the underlying graph of the associated Z-splitting of $\pi_1(X)$ is a tree. By Theorem 1.5(1),
this assumption depends only on the group $\pi_1(X)$.

(4) It is worth pointing out that the consideration in this paper provides an almost ideal setting
for the need for developing the algebraic theory of Rips and Sela on Z-splittings  
of single-ended finitely presented groups \cite{RipS}.

\vspace{2mm}

The next theorem, Theorem 1.6, is concerned with the building blocks of fiber-sum decompositions. 
In particular, it is shown that in most of the cases the diffeomorphism class of an irreducible 
$\s^1$-four-manifold is determined by the fundamental group. To state the result, we remark that a 
finitely generated group with infinite center is either single-ended or double-ended, cf. Lemma 4.1.

\begin{theorem}
Let $X, X^\prime$ be irreducible $\s^1$-four-manifolds, and let 
$\alpha: \pi_1(X)\rightarrow \pi_1(X^\prime)$ be any isomorphism. 
\begin{itemize}
\item [{(1)}] If $\pi_1(X), \pi_1(X^\prime)$ are single-ended, then there exists a diffeomorphism 
$\phi: X\rightarrow X^\prime$ such that $\phi_\ast =\alpha: \pi_1(X)\rightarrow \pi_1(X^\prime)$.
\item [{(2)}] If $\pi_1(X), \pi_1(X^\prime)$ are double-ended, then $X, X^\prime$ are the 
mapping-torus of a periodic diffeomorphism of an elliptic $3$-manifold. Moreover, 
there exists a diffeomorphism 
$\phi: X\rightarrow X^\prime$ such that $\phi_\ast =\alpha: \pi_1(X)\rightarrow \pi_1(X^\prime)$, if
the elliptic $3$-manifold is not a lens space.
\end{itemize}
\end{theorem}

Finally, the cases which are not covered in Theorems 1.5 and 1.6, i.e., when $\pi_1(X)$ is isomorphic 
to the fundamental group of a $2$-torus or a Klein bottle, are handled separately. In particular, we direct
the reader's attention to two classification theorems of fixed-point free $\s^1$-four-manifolds. One is concerned with the situation where the center of $\pi_1$ is of rank greater than $1$, the other is about 
the situation where $\pi_1$ is isomorphic to the $\pi_1$ of a Klein bottle. See Theorems 4.3 and 6.2 
for more details.

With the preceding understood, Theorem 1.1 follows readily from Theorems 1.5 and 1.6. Theorem 1.2 also follows from Theorems 1.5 and 1.6 with the additional help of Theorems 1.4, 4.3 and 6.2. 

\vspace{2mm}

Having reviewed the main theorems, we now give a few remarks about the technical aspect
of this paper. Our arguments rely heavily on the recent advances in $3$-dimensional topology, 
particularly those centered around the resolution of Thurston's Geometrization Conjecture 
(henceforth referred to as the {\it Geometrization Theorem}, cf. \cite{BLP, P}, see also \cite{DL}). 
For instance, Lemma 5.2, which asserts that an orientable $3$-orbifold is Seifert fibered 
if $\pi_1^{orb}$ has infinite center, and furthermore, if $\pi_2^{orb}\neq 0$, it is the mapping torus of
a periodic diffeomorphism of a $2$-orbifold with finite $\pi_1^{orb}$, played a key role in the proofs 
of several theorems of this paper. The proof of this lemma involves several particular forms 
of the Geometrization Theorem, which include the earlier work of Meeks and Scott \cite{MS} 
on finite group actions on Seifert $3$-manifolds, the resolution of the Seifert Fiber Space 
Conjecture due to Gabai \cite{Gabai} (and independently Casson-Jungreis \cite{CJ}), as 
well as the more recent Orbifold Theorem of Boileau, Leeb and Porti \cite{BLP} and the 
resolution of Poincar\'{e} conjecture (cf. \cite{P}). On the other hand, as we mentioned earlier 
this paper also draws considerably from geometric group theory, particularly the work of Rips 
and Sela on Z-splittings of single-ended finitely presented groups (cf. \cite{RipS}). 

\vspace{2mm}

Before ending the introduction, we point out a corollary of Theorem 1.4 which is of independent interest.

\begin{corollary}
Let $X$ be a $4$-manifold whose fundamental group has infinite center. If $X$ admits a locally linear,
fixed-point free $\s^1$-action, then there are no embedded $2$-spheres with odd
self-intersection in $X$. In particular, $X$ is minimal. 
\end{corollary}

We end the introduction with the following questions, which are naturally suggested by the results 
of this paper (see \cite{Sw,Thom,Turaev, MT, Turaev1} for some relevant problems and results in 
dimension three). 

\begin{question}
Let $X$ be an oriented, smooth, fixed-point free $\s^1$-four-manifold whose fundamental group 
has infinite center.

(1) Is the diffeomorphism type of $X$ determined by its homeomorphism type?

(2) Can one express the Seiberg-Witten invariant of $X$ in terms of topological invariants 
of the manifold?
\end{question} 

\vspace{3mm}

The organization of the rest of the paper is as follows. In Section 2, we first review some
basic definitions and facts about $2$-orbifolds and $3$-orbifolds, and then we prove several
preliminary lemmas which will be used in later sections. Section 3 is devoted to the proof of 
Theorem 1.5, which begins with a brief review of the Bass-Serre theory of groups acting 
on trees (in particular, the definition of graph of groups and its fundamental group), as well 
as a review on the relevant part of the work of Rips and Sela in \cite{RipS} concerning 
Z-splittings of single-ended finitely presented groups. The proof of Theorem 1.6 is given in 
Section 4, so is the classification of fixed-point free $\s^1$-four-manifolds 
whose $\pi_1$ has a center of rank greater than $1$. Section 5 is devoted to 
Theorem 1.4; in particular, we prove the key lemma, Lemma 5.2, in this section. Corollary 1.7 
asserting minimality of injective $\s^1$-four-manifolds is proven here as well. Section 6 contains 
the proofs of Theorems 1.1 and 1.2, as well as the classification of fixed-point free 
$\s^1$-four-manifolds whose $\pi_1$ is isomorphic to the $\pi_1$ of a Klein bottle.

Throughout this paper, we shall adopt the following notation: the center of a group $G$ is 
denoted by $z(G)$.

\section{Recollections and preliminary lemmas}

For the reader's convenience, we shall begin by giving a brief review on the relevant definitions and 
basic facts about $2$-orbifolds and $3$-orbifolds (for more details, see e.g. \cite{Sct, BMP}). Recall 
first that an orbifold (not necessarily orientable) is called {\it good} if it is the quotient of a manifold by a properly discontinuous action of a discrete group; otherwise it is call {\it bad}. It is called {\it very good} 
if it is the quotient of a manifold by a finite group action. All orbifolds are assumed to be connected and closed (i.e., compact without boundary) unless mentioned otherwise.

An orientable $2$-orbifold is given by a closed orientable surface as the underlying space, with isolated singular points where the local groups are cyclic, generated by a rotation.  
For a non-orientable $2$-orbifold, if the underlying space has a nonempty boundary, 
the singular set will also contain the boundary of the underlying space, which is a polygon
with local groups being either a reflection through a line in $\R^2$ or a dihedral group $D_{2n}$ generated by two reflections through lines making an angle $\pi/n$. With this understood, a 
{\it teardrop} is a $2$-sphere with one singular point. A {\it spindle} is a $2$-sphere with two 
singular points of different multiplicities (i.e., the orders of the local groups). A {\it football}
is a $2$-sphere with two singular points of the same multiplicity. A {\it turnover} is a $2$-sphere with 
three singular points. Except for a teardrop or a spindle, all orientable $2$-orbifolds are very good.
An orientable $2$-orbifold is called {\it spherical} (resp. {\it toric}, resp. {\it hyperbolic}) if it is the quotient of a $2$-sphere (resp. $2$-torus, resp. closed surface of genus $>1$) by a finite group. 
A $2$-orbifold is spherical if and only if it is either a nonsingular sphere, a football, or a turnover with
multiplicities $(2,2,n)$, $(2,3,3)$, $(2,3,4)$, or $(2,3,5)$. The turnovers correspond to the quotient of 
$2$-sphere by the action of a dihedral group $D_{2n}$ or one of the platonic groups $T_{12}$, $O_{24}$, $I_{60}$. 

All $2$-suborbifolds in a $3$-orbifold are assumed to be orientable.  
There is a special class of $3$-orbifolds which are important for the considerations in this paper;
these are the $3$-orbifolds which does not contain any bad $2$-suborbifolds. It is a consequence of
the Geometrization Theorem (cf. \cite{BLP, MM}) that if a $3$-orbifold does not contain any bad 
$2$-suborbifolds, then it must be very good, i.e., it is the quotient of a $3$-manifold by a finite group 
action. For simplicity, we shall call such a $3$-orbifold {\it good}. 

An orientable $3$-orbifold (with or without boundary) is called {\it spherical} (resp. {\it discal}) if it is the quotient of the $3$-sphere (resp. $3$-ball) by a finite isometry group.  A good $3$-orbifold is called {\it irreducible} if every spherical $2$-suborbifold bounds a discal $3$-orbifold. An irreducible $3$-orbifold is called {\it atoroidal} if it contains no essential toric $2$-suborbifold. A $3$-orbifold (not necessarily 
orientable) is called {\it Seifert fibered} if it is the total space of an orbifold bundle over a $2$-orbifold (not necessarily orientable) with generic fiber a circle or a mirrored interval. (A mirrored interval is the quotient of a circle by an orientation-reversing involution.) It is easily seen that a generic fiber of an orientable Seifert fibered $3$-orbifold must be a circle. Moreover, if the base $2$-orbifold is orientable, then the singular set 
of the Seifert fibered $3$-orbifold must consist of a union of fibers. 

The rest of this section is occupied by a number of preliminary lemmas.
The following lemma about the center of an amalgamated product or an HNN extension 
is well known to the experts. However, for the sake of completeness, we include a statement and
a proof of the lemma here. 

\begin{lemma}
(1) If $A\neq C\neq B$, then the center of $A \ast_C B$ is contained in $C$. 

(2) Let $C\subset A$ be a subgroup and $\alpha:C\rightarrow A$ be an injective homomorphism, 
and let $A\ast_C\alpha$ denote the corresponding HNN extension. Suppose $x\in z(A\ast_C\alpha)$. 
Then either $x\in C$, or $x$ is non-torsion, $C=A=\alpha(C)$, and $A\ast_C\alpha$ is isomorphic to 
$A\ast_C\alpha^\prime$ for some $\alpha^\prime: A\rightarrow A$ which is of finite order.
\end{lemma}

\begin{proof}
For a proof of part (1), see Magnus-Karrass-Solitar \cite{MKS}, Corollary 4.5, page 211. We shall give
a proof for part (2) here. An element of $A\ast_C\alpha$ can be uniquely represented by a reduced 
word (cf. e.g. Scott-Wall \cite{SW}). Lemma 2.1 is a direct consequence of this fact. 

More concretely, recall that the group $A\ast_C\alpha$ is generated by elements of $A$ and a letter 
$t$ with additional relations $tct^{-1}=\alpha(c)$ for all $c\in C$. We let $T$, $T_\alpha$ be the set 
of some fixed choices of representatives of the right cosets of $C$ and $\alpha(C)$ in $A$ respectively. 
Then a reduced word in $A\ast_C\alpha$ takes the following form
$$
a_1t^{\epsilon_1}a_2t^{\epsilon_2}\cdots a_nt^{\epsilon_n}a_{n+1},
$$
where $\epsilon_i=\pm 1$, $a_i\in T$ if $\epsilon_i=+1$, $a_i\in T_\alpha$ if $\epsilon_i=-1$, and 
furthermore, $a_i\neq 1$ if $\epsilon_{i-1}\neq \epsilon_i$, and $a_{n+1}$ is allowed to be an arbitrary element of $A$.

Let $x=a_1t^{\epsilon_1}a_2t^{\epsilon_2}\cdots a_nt^{\epsilon_n}a_{n+1}$
be an element of the center (here $n=0$ represents the case where $x\in A$). If $n=0$, then by $tx=xt$
it is clear that $x=a_{n+1}\in C$ which obeys $\alpha(x)=x$. Suppose $n>0$. If $a_1\neq 1$, then the
uniqueness of representation by reduced words implies that $tx\neq xt$, which is a contradiction. 
If $a_1=1$, then $t^{-\epsilon_1}x=xt^{-\epsilon_1}$ implies that $a_2=1$. Iterating this process,
we see that $x=t^l a_{n+1}$ for some $0\neq l\in\Z$. It follows from $t^{-1}x=xt^{-1}$ that 
$a_{n+1}=ta_{n+1}t^{-1}$, which implies that $a_{n+1}\in C$ and $\alpha(a_{n+1})=a_{n+1}$.
Furthermore, the commutativity of $t$ and $a_{n+1}$ also implies that $x=t^l a_{n+1}$ is non-torsion. 
To see $C=A=\alpha(C)$, note that if there is an $a\in T$ or $T_\alpha$ such that 
$a\neq 1$, then one has $ax\neq xa$ which is a contradiction. This implies that $C=A=\alpha(C)$.  
Now for any $c\in C$,
$$
t^l a_{n+1} \alpha^l(c)=x\alpha^l(c)=\alpha^l(c)x=\alpha^l(c)t^la_{n+1}=t^l ca_{n+1},
$$
which implies that $a_{n+1} \alpha^l(c)=ca_{n+1}$ for any $c\in C=A$. Let $\alpha^\prime:
A\rightarrow A$ be defined by $\alpha^\prime(c)=a_{n+1}\alpha(c)a_{n+1}^{-1}$. Then it follows from
$\alpha(a_{n+1})=a_{n+1}$ that 
$$
(\alpha^\prime)^l(c)=a_{n+1}\alpha^l(c)a_{n+1}^{-1}=c, \;\;\forall c\in A.
$$
Now note that $A\ast_C\alpha$ is isomorphic to $A\ast_C\alpha^\prime$ where $\alpha^\prime$
has finite order $l$. This completes the proof of the lemma.

\end{proof}

For our purposes in this paper, it is important to understand the center of the 
fundamental group of a $2$-orbifold or a $3$-orbifold.

\begin{lemma}
Let $\Sigma$ be a $2$-orbifold (not necessarily orientable) such that $z(\pi_1^{orb}(\Sigma))$
is nontrivial. Then the following statements hold true.
\begin{itemize}
\item [{(a)}] If $\Sigma$ is orientable, then it is either a football, a spindle with non-coprime multiplicities, a turnover with multiplicities $(2,2,2)$, or a nonsingular torus.
\item [{(b)}] If $\Sigma$ is non-orientable, then its orientable double cover $\tilde{\Sigma}$ must lie 
in the following list: a nonsingular sphere, a teardrop, a spindle, a football, a turnover with multiplicities $(2,2,2)$, or a nonsingular torus. Moreover, $z(\pi_1^{orb}(\Sigma))$ is torsion-free 
if and only if $\tilde{\Sigma}$ is a nonsingular torus. 
 \end{itemize}
\end{lemma}

\begin{proof}
Suppose $\Sigma$ is orientable. If $\Sigma$ is bad, then it must be a spindle with non-coprime 
multiplicities because this is the only case where $\pi_1^{orb}(\Sigma)$ is nontrivial. 
Assume $\Sigma$ is good. If $\Sigma$ is spherical, then it must be a football or a turnover with multiplicities $(2,2,2)$, because the other groups, i.e., $D_{2n}$ with $n\neq 2$, $T_{12}$, 
$O_{24}$, $I_{60}$, all have trivial center. If $\Sigma$ is toric, then it must be a nonsingular torus because the fundamental group of a toric turnover is centerless. Finally, $\Sigma$ can not be hyperbolic because a co-compact Fuchsian group has trivial center. 

Suppose $\Sigma$ is non-orientable, and let $\tilde{\Sigma}$ be the orientable $2$-orbifold which doubly covers $\Sigma$. Note that $\Z_2$ acts on $\tilde{\Sigma}$ via deck transformations. We shall discuss the proof according to (i) the deck transformations are free, (ii) the deck transformations are not free.

In case (i), the underlying space $|\Sigma|$ is a non-orientable, closed surface. We can decompose 
$|\Sigma|$ as the union of $\R\P^2\setminus D^2$ and an orientable surface with one boundary component along their boundaries. Correspondingly, we have a decomposition of $\Sigma$ as the union of (nonsingular) $\R\P^2\setminus D^2$ and an orientable $2$-orbifold $\Sigma^\prime$ with
one boundary component. It follows that $z(\pi_1^{orb}(\Sigma))$ being nontrivial forces $\pi_1^{orb}(\Sigma^\prime)$ to be finite (cf. Lemma 2.1(1) and Lemma 2.2(a)), so that $\Sigma^\prime$ must be either a (nonsingular) $D^2$ or $D^2/\Z_m$ with $m>1$. This shows that the double cover $\tilde{\Sigma}$ is either a (nonsingular) sphere or a football.

In case (ii), if $z(\pi_1^{orb}(\tilde{\Sigma}))$ is nontrivial, then we are done by part (a).
Moreover, if $\tilde{\Sigma}$ is a nonsingular torus, the fixed-point set of the deck transformation consists of a union of circles. Since the deck transformation is orientation-reversing, the Lefschetz fixed-point theorem implies that the action on $H_1(\tilde{\Sigma};\R)$ must have eigenvalues $+1$ and $-1$. It follows then that $z(\pi_1^{orb}(\Sigma))=\Z$ in this case. If 
$z(\pi_1^{orb}(\tilde{\Sigma}))$ is trivial, then $z(\pi_1^{orb}(\Sigma))=\Z_2$ and acts on 
$\tilde{\Sigma}$ via deck transformations. Let $p$ be a fixed-point of the deck transformation. 
Since $\pi_1^{orb}(\tilde{\Sigma})\rightarrow \pi_1(|\tilde{\Sigma}|)$ 
is surjective, the induced action of $z(\pi_1^{orb}(\Sigma))=\Z_2$ on $\pi_1(|\tilde{\Sigma}|,p)$ must be trivial. This implies that the Lefschetz number of the action of $z(\pi_1^{orb}(\Sigma))=\Z_2$ on 
$|\tilde{\Sigma}|$ equals $-2$ times the genus of $|\tilde{\Sigma}|$. The Lefschetz fixed-point 
theorem then implies that $|\tilde{\Sigma}|$ has genus zero. If $\tilde{\Sigma}$ is bad, then clearly 
we are done. If $\tilde{\Sigma}$ is good, then it is the quotient of an orientable closed surface 
$\Sigma^\prime$ by a finite group. Note that $z(\pi_1^{orb}(\Sigma))=\Z_2$ also acts on 
$\Sigma^\prime$ via deck transformations which is orientation-reversing. The same argument as above shows that $\Sigma^\prime$ must have genus zero. In other words, $\tilde{\Sigma}$ is spherical. It follows easily that it must be either a (nonsingular) sphere, a football or a turnover with multiplicities $(2,2,2)$. (In fact $\tilde{\Sigma}$ is a sphere because we assume $z(\pi_1^{orb}(\tilde{\Sigma}))$ is trivial.) Hence the lemma.

\end{proof}

\begin{lemma}
Let $Y$ be an irreducible $3$-orbifold with infinite $\pi_1^{orb}(Y)$. Then $z(\pi_1^{orb}(Y))$
is torsion-free.
\end{lemma}

\begin{proof}
By the JSJ-decomposition theorem for $3$-orbifolds (cf. \cite{BMP}, Theorem 3.3), there is a finite
collection (possibly empty) of disjoint, essential toric $2$-suborbifolds $\Sigma_j$, $j=1,2,\cdots,m$, 
which split $Y$ into $3$-suborbifolds $Y_i$, $i=1,2,\cdots,n$, such that each $Y_i$ is either Seifert fibered or atoroidal. This presents $\pi_1^{orb}(Y)$ as the fundamental group of a finite graph of groups, where the vertex groups are $\pi_1^{orb}(Y_i)$ and the edge groups are $\pi_1^{orb}(\Sigma_j)$. If $\{\Sigma_j\}$ is not empty, then the torsion part of $z(\pi_1^{orb}(Y))$ must lie in the edge groups 
$z(\pi_1^{orb}(\Sigma_j))$ (cf. Lemma 2.1). By Lemma 2.2(a), 
$z(\pi_1^{orb}(\Sigma_j))$ is torsion-free, which implies that $z(\pi_1^{orb}(Y))$ is torsion-free 
when $\{\Sigma_j\}$ is not empty.

Suppose $\{\Sigma_j\}$ is empty. Then $Y$ is either Seifert fibered or atoroidal. Assume $Y$ is 
Seifert fibered first, and let $\pi: Y\rightarrow B$ be a Seifert fibration. There is an induced exact sequence (cf. \cite{BMP}, Proposition 2.12)
$$
1 \rightarrow C \rightarrow\pi_1^{orb}(Y)\stackrel{\pi_\ast}{\rightarrow}\pi_1^{orb}(B)\rightarrow 1,
$$
where $C$ is cyclic or dihedral (either finite or infinite). In addition, $C$ is finite if and only if 
$\pi_1^{orb}(Y)$ is finite. At the present case $Y$ has only $1$-dimensional singular set, so that 
a generic fiber of $\pi$ must be a circle. Consequently, $C$ is cyclic in the above exact sequence. Since 
$\pi_1^{orb}(Y)$ is infinite, we have $C=\Z$. On the other hand, $C=\pi_1(\s^1)/
\text{Image }\delta$, where $\delta:\pi_2^{orb}(B)\rightarrow \pi_1(\s^1)$ is the connecting homomorphism in the exact sequence of homotopy groups associated to the Seifert fibration $\pi: Y\rightarrow B$. 
(For the definition of homotopy groups of orbifolds and the exact sequence associated to an orbifold fibration, see \cite{Hae1, Hae2, C0}.) As $C$ is infinite, $\delta$ must be the zero map, and consequently, 
$\pi_\ast:\pi_2^{orb}(Y)\rightarrow \pi_2^{orb}(B)$ is surjective. By the assumption that 
$Y$ is irreducible, its universal cover $\tilde{Y}$ is also irreducible (cf. \cite{BMP}, Theorem 3.23).
Consequently, we have $\pi_2^{orb}(Y)=\pi_2(\tilde{Y})=0$ which implies that 
$\pi_2^{orb}(B)=0$. Now observing that a bad $2$-orbifold and a spherical $2$-orbifold must have 
nontrivial $\pi_2^{orb}$, we conclude, by Lemma 2.2, that $z(\pi_1^{orb}(B))$ must be torsion-free.
It follows easily that $z(\pi_1^{orb}(Y))$ is torsion-free in this case. 

It remains to consider the case where $Y$ is atoroidal. If $Y$ is nonsingular (i.e., a $3$-manifold), 
then $\pi_1^{orb}(Y)=\pi_1(Y)$ is torsion-free, hence $z(\pi_1^{orb}(Y))$ must be torsion-free.
If $Y$ is an honest orbifold, then by the Orbifold Theorem of Boileau, Leeb and Porti
(cf. \cite{BLP}, Corollary 1.2), $Y$ is 
geometric. In fact, we will need the following more precise statement: $Y$ has a metric of constant 
curvature or is Seifert fibered. It is clear that, since $\pi^{orb}_1(Y)$ is infinite, we only need to
discuss the following two cases: (i) $Y$ is hyperbolic, (ii) $Y$ is Euclidean.

Suppose $Y$ is hyperbolic. Then there is a hyperbolic $3$-manifold $Y^\prime$ and a finite
group of isometries $G$ such that $Y=Y^\prime/G$. Now suppose $z(\pi_1^{orb}(Y))$ is not 
torsion-free, and let $g\in z(\pi_1^{orb}(Y))$ be a torsion element. Then since $\pi_1(Y^\prime)$ is torsion-free, $g$ may be regarded as an element of $G$ and acts on $Y^\prime$ via deck transformations. Moreover, $g$ must have a fixed point, say $p\in Y^\prime$. This gives rise
to an automorphism $g_\ast$ of $\pi_1(Y^\prime,p)$, which is trivial because $g\in z(\pi_1^{orb}(Y))$.
By Mostow Rigidity, $g:Y^\prime\rightarrow Y^\prime$ is trivial, which is a contradiction. 

Suppose $Y$ is Euclidean. By Bieberbach Theorem (cf. \cite{Sct}, p. 443), $Y$ is finitely covered
by $T^3$ with deck transformation group $G$. Let $x\in z(\pi_1^{orb}(Y))$ be a torsion element. 
Then $x$ may be regarded as an element of $G$ and acts on $T^3$ via deck transformations.
Furthermore, $x$ must have a fixed point, say $p\in T^3$. Since $x$ is central, the induced
automorphism $x_\ast: \pi_1(T^3,p)\rightarrow\pi_1(T^3,p)$ must be trivial. It follows that $x$
is trivial, which is a contradiction. 

This completes the proof of the lemma.

\end{proof}

Given any good $3$-orbifold $Y$ which is not irreducible, one can cut $Y$ open 
along a finite system of spherical $2$-suborbifolds into pieces which are irreducible. 
More precisely, by the spherical decomposition theorem (cf. Theorem 3.2, \cite{BMP}), there 
is a finite, nonempty collection of disjoint spherical 2-suborbifolds $\{\Sigma_j\}$ such that 
each component $Y_i$ of $Y\setminus \{\Sigma_j\}$ becomes an irreducible $3$-orbifold 
after capping-off the boundary spherical $2$-suborbifolds by the corresponding discal 
$3$-orbifolds. 

For the purpose in this paper, a slightly improved version of the above statement 
is needed. More concretely, given any system of spherical $2$-suborbifolds 
$\{\Sigma_j\}$ of $Y$, let $\{Y_i\}$ be the set of components of $Y\setminus \{\Sigma_j\}$. 
We say that $\Sigma_j$ is separating (resp. non-separating) in $Y_i$ if $\Sigma_j$ is a boundary 
component (resp. a non-separating spherical $2$-suborbifold) of the closure of $Y_i$ in $Y$. 
(Note that $\Sigma_j$ can be a non-separating spherical $2$-suborbifold of $Y$ but is 
separating in $Y_i$.)
With this understood, we say that the corresponding spherical decomposition of Y is {\it reduced} 
if for any $\Sigma_j$, $Y_i$ such that $\Sigma_j$ is separating in $Y_i$, $\pi_1^{orb}(\Sigma_j)$ 
is a proper subgroup of $\pi_1^{orb}(Y_i)$ under the inclusion of $\Sigma_j$ in the closure of 
$Y_i$ in $Y$. 

\begin{lemma}
For any good, non-irreducible $3$-orbifold $Y$, there exists a reduced spherical 
decomposition of $Y$ into irreducible $3$-orbifolds. 
\end{lemma}

\begin{proof}
Given any spherical decomposition of $Y$ into irreducible pieces which always 
exists (cf. Theorem 3.2, \cite{BMP}), we can modify it into a reduced spherical decomposition 
as follows. Let $\{\Sigma_j\}$ be the corresponding system of spherical 2-suborbifolds and 
let $\{Y_i\}$ be the set of components of $Y\setminus \{\Sigma_j\}$. Suppose for some $i,j$, 
$\Sigma_j$ is separating in $Y_i$ and $\pi_1^{orb}(\Sigma_j)=\pi_1^{orb}(Y_i)$. Let $Y_k\in\{Y_i\}$ 
be the other component whose closure in $Y$ also contains $\Sigma_j$ as a boundary component. 
Then observe that the $3$-orbifold obtained from capping-off $Y_k\cup\Sigma_j\cup Y_i$ is the same 
as that obtained from capping-off $Y_k$. This is because by the Geometrization Theorem, 
the $3$-orbifold obtained from capping-off the boundary components of $Y_i$ other than $\Sigma_j$ 
is a discal $3$-orbifold with boundary $\Sigma_j$. Consequently, if we remove $\Sigma_j$ from 
$\{\Sigma_j\}$, the corresponding spherical decomposition still splits $Y$ into irreducible pieces. 
Continuing this process, we arrive at a reduced spherical decomposition in finitely many steps. 
Hence the lemma. 

\end{proof}

We remark that given any spherical decomposition of a good $3$-orbifold $Y$, with 
$\{\Sigma_j\}$ being the system of spherical $2$-suborbifolds and $\{Y_i\}$
being the set of components of $Y\setminus \{\Sigma_j\}$,
one has a corresponding finite graph of groups whose vertex groups and 
edge groups are given by $\{\pi_1^{orb}(Y_i)\}$ and $\{\pi_1^{orb}(\Sigma_j)\}$
respectively, such that $\pi_1^{orb}(Y)$ 
is naturally isomorphic to the fundamental group of the graph of groups. When the 
spherical decomposition is reduced, the corresponding graph of groups is also reduced 
in the sense that an edge group is always a proper subgroup of the vertex groups 
as long as the end points of the edge are distinct vertices. Given any finite graph of 
groups, one can always modify it into a reduced one without changing the isomorphism 
class of the fundamental groups by collapsing a number of edges. Lemma 2.4 is simply 
a manifestation of this principle in the geometric setting of spherical decomposition of 
$3$-orbifolds. When there are no non-separating spherical $2$-suborbifolds, the existence 
and uniqueness of reduced spherical decompositions were proven in \cite{Pe} 
(called {\it efficient splittings} therein). 

Next we give a classification of certain orientation-preserving finite group actions on 
$\s^1\times \s^2$. The case where the actions are free or have only isolated exceptional orbits 
was discussed in Meeks-Scott \cite{MS}, Theorem 8.4. Our discussion relies on the 
Equivariant Sphere Theorem of Meeks and Yau (cf. \cite{MY}) and Geometrization of
finite group actions on $\s^3$ (compare also Dinkelbach-Leeb \cite{DL} via equivariant Ricci flow). 

In order to state the result, we shall fix the following convention and notations. We orient $\s^3$ as the boundary of the unit ball in $\C^2$, and consider certain orientation-preserving $\Z_{2m}$-actions on 
$\s^3$. When $m$ is even there is only one such action up to a change of generators of $\Z_{2m}$.
When $m$ is odd, there are two non-equivalent such actions, and we shall denote the quotient orbifolds
by $\R\P^3_m$, $\widetilde{\R\P^3}_m$ respectively. More concretely, we fix a generator $t$
of $\Z_{2m}$, and let 
$$
\R\P^3_m=\s^3/\Z_{2m}, \mbox{ where } t\cdot (z_1,z_2)=(-z_1,\exp(\frac{\pi i}{m}) z_2), 
$$
and 
$$
\widetilde{\R\P^3}_m=\s^3/\Z_{2m}, \mbox{ where } t\cdot (z_1,z_2)
=(-z_1,\exp(\frac{(m+1)\pi i}{m}) z_2), \;\; m \mbox { is odd. }
$$
Note that when $m>1$, these actions can be characterized by the fact that the whole group has no
fixed points but the index $2$ subgroup fixes an unknotted circle. Moreover, 
the difference between $\R\P^3_m$ and $\widetilde{\R\P^3}_m$ is that the singular
set of $\widetilde{\R\P^3}_m$ has two components, of multiplicities $2$ and $m$ respectively,
while the singular set of $\R\P^3_m$ has only one component, of multiplicity $m$.

\begin{lemma}
Let $G$ be a finite group which acts on $\s^1\times \s^2$ preserving the orientation. 
\begin{itemize}
\item Suppose the action of $G$ is homologically trivial. Then $\s^1\times \s^2/G$ is the mapping
torus of a periodic diffeomorphism of some spherical $2$-orbifold.
\item Suppose $G$ is cyclic and is generated by $t$ which is homologically non-trivial. Then 
the quotient orbifold $\s^1\times \s^2/G$ is diffeomorphic to one of the following
$$
\R\P^3_m \#_{m} \R\P^3_m,\;\;\; \R\P^3_m \#_{m} \widetilde{\R\P^3}_m, \;\; \mbox{ or  }
\widetilde{\R\P^3}_m\#_{m} \widetilde{\R\P^3}_m,
$$
where $\#_m$ denotes the connected sum of orbifolds over a point of multiplicity $m$, such that a 
generator of the $\pi_1^{orb}$ of $\s^2/\Z_m$ has the same image on both sides of the connected sum.
\end{itemize}
\end{lemma}

\begin{proof}
First of all, by the Equivariant Sphere Theorem of Meeks and Yau (cf. \cite{MY}, page 480), there exists a 
finite set of embedded $2$-spheres $\{\Sigma_i\}$ of $\s^1\times\s^2$ which is $G$-invariant and 
generates the $\pi_2$ as a $\pi_1$-module. Since $\pi_2(\s^1\times\s^2)$ has rank $1$, we may 
assume $G$ acts on the set of spheres $\{\Sigma_i\}$ transitively. It follows easily from the 
Geometrization Theorem that when cutting $\s^1\times\s^2$ open along the $\Sigma_i$'s, 
each component $Y_j$ of $\s^1\times\s^2\setminus \{\Sigma_i\}$ is a 
$3$-manifold diffeomorphic to the product of $\s^2$ with an interval. 

For convenience of the argument, we shall consider the following finite graph $\Gamma$, where the
vertices correspond to the components $Y_j$ and the edges to the embedded spheres $\Sigma_i$,
and $\Sigma_i$ is incident to $Y_j$ if and only if $\Sigma_i$ is contained in the closure of $Y_j$. 
Clearly $\Gamma$ is homeomorphic to a circle, and there is an induced simplicial action of $G$ on 
$\Gamma$. We denote by $G_0$ the subgroup of $G$ which acts trivially on $\Gamma$. 

Suppose $G_0$ is non-trivial. We pick an embedded sphere $\Sigma_i$ and cut $\s^1\times\s^2$
open along $\Sigma_i$. Because $\Sigma_i$ is $G_0$-invariant, we can close up 
$\s^1\times\s^2\setminus \Sigma_i$ and obtain a $G_0$-action on $\s^3$. By the Geometrization
Theorem, the action of $G_0$ is given by an isometry, which implies that  the original 
$G_0$-action on $\s^1\times \s^2$ is a product action which is trivial on the $\s^1$-factor. 
Note that we are done if $G=G_0$.

Assume $G\neq G_0$ and consider the action of $G$. In the case where $G$ acts homologically 
trivially, $G/G_0$ acts effectively on $\Gamma$ by rotations. This implies that $\s^1\times\s^2/G$ 
is the mapping torus of the $2$-orbifold $\s^2/G_0$ for some periodic diffeomorphism of $\s^2/G_0$ 
which generates $G/G_0$. The lemma follows easily in this case.

Suppose $G$ is generated by $t$ which is homologically non-trivial. Then the induced action 
of $t$ on $\Gamma$ must be given by a reflection, and $G_0$ is an index $2$ subgroup. Furthermore, $G_0$ is cyclic in this case and the action of $G_0$ on $\s^2$ is given by rotations. 
The order of $t$ is even, say $2m$, and there are two possibilities for the induced action of $t$ on the 
graph $\Gamma$: (i) $t$ has an invariant edge, (ii) $t$ fixes two vertices. 

In case (i) $t$ leaves an embedded sphere $\Sigma_i$ invariant (which is the only one because 
by assumption $G$ acts transitively on the set of spheres $\{\Sigma_i\}$). The induced action
of $t$ on $\Sigma_i$ is orientation-reversing, and there are two non-equivalent actions
when $m$ is odd. More concretely, if we identify $\Sigma_i$ with the unit sphere in 
$\R^3=\R\times\C$, then the actions are given by
$$
 t\cdot (x,z)=(-x,\exp(\frac{\pi i}{m}) z),  \mbox{ where } (x,z)\in \R\times \C,
$$
and
$$
 t\cdot (x,z)=(-x,\exp(\frac{(m+1)\pi i}{m}) z),  \mbox{ where } (x,z)\in \R\times \C, \; m \mbox { is odd. }
$$
It follows easily that the quotient of a $t$-invariant regular neighborhood of $\Sigma_i$ is 
diffeomorphic to either $\R\P^3_m$ or $\widetilde{\R\P^3}_m$ with a ball centered at a 
singular point of multiplicity $m$ removed. Moreover, the complement of the $t$-invariant 
regular neighborhood is a $3$-manifold $Y_j$ which is diffeomorphic to the product of $\s^2$ 
with an interval. The action of $t$ on $Y_j$ can be naturally extended to an $t$-action on $\s^3$
by capping-off the boundary of $Y_j$, which, by the Geometrization Theorem, is equivalent to
an isometry. Note that when $m>1$, $t^2$ has a $1$-dimensional fixed-point set. It follows easily
that $Y_j/\langle t\rangle$ is also diffeomorphic to either $\R\P^3_m$ or $\widetilde{\R\P^3}_m$ 
with a ball centered at a singular point of multiplicity $m$ removed, and $\s^1\times\s^2/G$ is 
diffeomorphic to either $\R\P^3_m \#_{m} \R\P^3_m$, or $\R\P^3_m \#_{m} \widetilde{\R\P^3}_m$, 
or $\widetilde{\R\P^3}_m\#_{m} \widetilde{\R\P^3}_m$ as claimed. 

In case (ii) where $t$ fixes two vertices of the graph $\Gamma$, the set $\{\Sigma_i\}$ has two elements
$\Sigma_1$, $\Sigma_2$, and $\s^1\times\s^2\setminus \{\Sigma_i\}$ has two components $Y_1,Y_2$,
such that $Y_1,Y_2$ are $t$-invariant and $t$ switches $\Sigma_1$ and $\Sigma_2$. Similarly, 
the $t$-actions on $Y_1,Y_2$ can be extended to a $t$-action on $\s^3$ by capping-off the boundary, 
and by the Geometrization Theorem, the quotient of $Y_1,Y_2$ by $t$ is diffeomorphic to either 
$\R\P^3_m$ or $\widetilde{\R\P^3}_m$, and the lemma follows in this case too. 

\end{proof}

We end with a lemma concerning existence of Seifert-type $T^2$-fibrations on a $4$-manifold.

\begin{lemma}
Let $\pi: X\rightarrow Y$ be a principal $\s^1$-bundle over an orientable $3$-orbifold where 
$Y$ is Seifert fibered. If the homotopy class of a regular fiber of the Seifert fibration on $Y$ lies 
in the image of $z(\pi_1(X))$ under $\pi_\ast: \pi_1(X)\rightarrow \pi_1^{orb}(Y)$, then 
$\pi:X\rightarrow Y$ may be extended to a principal $T^2$-bundle over a $2$-orbifold. 
\end{lemma}

\begin{proof}
Let $pr:Y\rightarrow B$ be the Seifert fibration on $Y$ where $B$ is a $2$-orbifold. (We note 
that $B$ must be orientable because the class of a regular fiber of $pr$ lies in the center $z(\pi_1^{orb}(Y))$.) Then the composition of $\pi$ with $pr$, $\Pi:X\rightarrow B$, defines $X$ 
as a $T^2$-bundle over $B$. We shall prove that $\Pi$ is principal, which is equivalent to the condition that $\Pi$ has a trivial monodromy representation. 

To see that the monodromy representation of $\Pi$ is trivial, we consider an arbitrary loop $\gamma$
in $B$ lying in the complement of the singular set. Pick a base point $b_0\in \gamma$, and a base
point $x_0\in \Pi^{-1}(b_0)$.  Choose a section $\gamma^\prime$ of $\Pi$ over $\gamma$ through 
$x_0$, and a loop $\delta$ containing $x_0$ in $X$ which is a section of $\pi$ over the fiber of $pr$
at $b_0$. Let $h$ be the fiber of $\pi$ containing $x_0$. With this understood, the monodromy representation of $\Pi$ is trivial if and only if the classes of $h$, $\delta$, and $\gamma^\prime$
in $\pi_1(X)$ commute. But this is clear because the class of both $h$ and $\delta$ lies in 
the center $z(\pi_1(X))$. Hence the lemma. 

\end{proof}

\section{Fiber-sum decomposition and fundamental group}

This section contains three subsections. Section 3.1 is devoted to a review of Bass-Serre theory and
Rips-Zela theory, and it also contains a proof of Lemma 3.1 and Lemma 3.2. Section 3.2 is occupied by a proof of Theorem 1.5(1), as given through Lemma 3.3, Lemma 3.4, and Proposition 3.5. The last
subsection, Section 3.3, contains the proof for Theorem 1.5(2). 

\subsection{Some recollections in geometric group theory}
We begin with a brief review of the Bass-Serre theory of groups acting on trees, see e.g. 
\cite{DD, SW} for more details. 

Let $\Gamma$ be a connected, nonempty graph, with the set of vertices and edges denoted by
$V\Gamma$ and $E\Gamma$ respectively, and the incidence functions denoted by $\iota,\tau:
E\Gamma\rightarrow V\Gamma$. Recall that a group of graphs, denoted by $G_\Gamma$,
consists of the following data: each $v\in V\Gamma$, $e\in E\Gamma$ is assigned with a group
$G(v)$, $G(e)$ respectively, and for each $e\in E\Gamma$ there is a pair of boundary monomorphisms 
$\alpha: G(e)\rightarrow G(\iota e)$ and $\omega: G(e)\rightarrow G(\tau e)$. 

Let $\Gamma_0$ be a maximal tree in $\Gamma$. The fundamental group of $G_\Gamma$ 
with respect to $\Gamma_0$, denoted by $\pi(G_\Gamma,\Gamma_0)$, is the group given by
the following presentation:
\begin{itemize}
\item generating set: $\{t_e|e\in E\Gamma\}\cup \bigcup_{v\in V\Gamma}G(v)$
\item relations: the relations for $G(v)$, $\forall v\in V\Gamma$, $t^{-1}_e\alpha(g)t_e=\omega(g)$,
$\forall g\in G(e)$, $\forall e\in E\Gamma$, and $t_e=1$, $\forall e\in E\Gamma_0=E\Gamma\cap\Gamma_0$.
\end{itemize}
It is known that the isomorphism class of $\pi(G_\Gamma,\Gamma_0)$ is independent of $\Gamma_0$,
which is called the {\it  fundamental group} of the graph of groups $G_\Gamma$.

Given any graph of groups $G_\Gamma$, there is a canonically constructed tree $T$, called the
{\it Bass-Serre tree}, together with a canonical action of the fundamental group of $G_\Gamma$. 
Moreover, the graph of groups $G_\Gamma$ can be recovered from the action of its fundamental 
group on the Bass-Serre tree in a canonical way, which we describe below.

Let $G$ be a group acting on a tree $T$ without inversion, i.e., the action sends vertices to vertices
and edges to edges, such that every edge invariant under the action is being fixed. 
Let $\Gamma$ be the quotient graph, and $p:T\rightarrow \Gamma$ be the quotient map. 
Let $T^\prime\subset T$ be a subset and $T_0\subset T^\prime$ be a subtree of $T$. 
We call $T^\prime$ a {\it fundamental $G$-transversal} in $T$ with subtree $T_0$, if (i) 
$p:T^\prime\rightarrow \Gamma$ is bijective, and (ii) $p:T_0\rightarrow\Gamma$ is onto a maximal tree in $\Gamma$. It is known that such a pair $(T^\prime,T_0)$ always
exists. Note that by (i), one can give a canonical graph structure to $T^\prime$ as follows: 
$VT^\prime=VT\cap T^\prime$, $ET^\prime=ET\cap T^\prime$, and the incidence functions 
$\bar{\iota},\bar{\tau}:ET^\prime\rightarrow VT^\prime$ are defined by the equations 
$$
p(\bar{\iota}e)=p(\iota e),\;\; p(\bar{\tau}e)=p(\tau e), \;\;\forall e\in ET^\prime. 
$$
(Here $\iota,\tau$ are the incidence functions of $T$.) Note that by (ii), $T_0$ is a maximal tree in 
$T^\prime$ with respect to this graph structure, and $\bar{\iota} e=\iota e$, $\bar{\tau} e=\tau e$ for
any $e\in ET_0$. 

Now given any fundamental $G$-transversal $T^\prime$ with subtree $T_0$, one can canonically
construct a graph of groups $G_\Gamma$ as follows, where $\Gamma$ and $T^\prime$ are identified 
as graphs. For any $v\in VT^\prime$, we assign to it the group $G(v)=G_v=\{g\in G|gv=v\}$, and
for any $e\in ET^\prime$, we assign to it the group $G(e)=G_e=\{g\in G|ge=e\}$. The boundary
monomorphisms $\alpha: G(e)\rightarrow G(\bar{\iota} e)$, $\omega: G(e)\rightarrow G(\bar{\tau}e)$ 
are defined as follows. For any $e\in ET^\prime$, pick $g_e,h_e\in G$ such that $g_e\bar{\iota}e=
\iota e$, $h_e\bar{\tau} e=\tau e$, where for any $e\in ET_0$, $g_e=h_e=1$. Then for any
$g\in G(e)$, define $\alpha(g)=g_e^{-1}gg_e$ and $\omega(g)=h_e^{-1}gh_e$ (note that
$G(e)\subset G(\iota e)$, $G(e)\subset G(\tau e)$). 

There is an obvious homomorphism $\phi: \pi(G_\Gamma,T_0)\rightarrow G$ which sends
$t_e$ to $g_e^{-1}h_e\in G$. The fundamental theorem of the Bass-Serre theory asserts that
$\phi$ is an isomorphism. Moreover, when $T$ is the Bass-Serre tree of a graph of groups 
$G_\Gamma$ and $G$ is the fundamental group of $G_\Gamma$ with the canonical action 
on $T$, the graph of groups $G_\Gamma$ can be recovered in the above manner. 

Next we review the Rips-Sela theory (see \cite{RipS} for more details). Given any group $G$,
a {\it Z-splitting} of $G$ is a presentation of $G$ as the fundamental group of a finite graph
of groups where all the edge groups are infinite cyclic. {\it Elementary Z-splittings} are
Z-splittings for which the graph of groups contains only one edge, i.e., an amalgamated 
product or an HNN extension. Given a Z-splitting of $G$ and an elementary Z-splitting of
a vertex group of the Z-splitting which is compatible with the boundary monomorphisms, there is a
naturally defined new Z-splitting of $G$ which is called an {\it elementary refinement}, where
the new graph of groups is obtained by replacing the vertex in the original graph by the 
corresponding one edge graph.  A {\it refinement} of a Z-splitting is the result of a sequence 
of elementary refinements. The inverse operation of a refinement is called a {\it collapse}. 

The fundamental result in the Rips-Sela theory concerns the existence of certain universal 
Z-splittings of a single-ended finitely presented group, called {\it canonical JSJ decompositions}, 
from which all other Z-splittings of the group can be derived in a certain organized way (involving
refinement or collapse). 
The starting point of this work is an analysis of the interactions between two distinct elementary 
Z-splittings. To be more concrete, let $G=A_i\ast_{C_i} B_i$ (or $A_i\ast_{C_i}$) be two given 
elementary Z-splittings, where $C_i$ is generated by $c_i$, for $i=1,2$. The element $c_2$ is 
called {\it elliptic} with respect to the first splitting if it is contained in a conjugate of $A_1$ or $B_1$, 
and {\it hyperbolic} otherwise, and similarly for $c_1$ with respect to the second splitting. With this
understood, one of the basic result in the Rips-Sela theory (cf. Theorem 2.1, \cite{RipS}) asserts 
that if $G$ is freely indecomposable, then $c_1$ and $c_2$ are simultaneously elliptic or 
simultaneously hyperbolic. 

The bulk of the Rips-Sela theory is devoted to the analysis of hyperbolic-hyperbolic splittings.
Our first observation is that for a group $G$ with infinite $z(G)$, hyperbolic-hyperbolic splittings
seldom occur, which greatly simplifies the situation. 

\begin{lemma}
Let $G$ be a single-ended group with infinite $z(G)$, which is not isomorphic to the fundamental
group of a $2$-torus or Klein bottle. Then (i) the center $z(G)$ is contained in the edge groups of
every reduced Z-splitting of $G$, and (ii) there are no hyperbolic-hyperbolic elementary Z-splittings 
of $G$.
\end{lemma}

\begin{proof}

We shall first prove part (i) of the lemma, where it suffices to consider only the case of elementary 
Z-splittings. Let $G=A\ast_C B$ or $A\ast_C$ be an elementary Z-splitting, where $A\neq C\neq B$.
By Lemma 2.1, if the splitting is an amalgamated product, then $C$ contains $z(G)$. If
the splitting is an HNN extension and $C$ does not contain $z(G)$, then $A=C=\langle c\rangle$ which
is infinite cyclic, and $G$ is isomorphic to $A\ast_A\alpha$ for a finite order automorphism
$\alpha$ of $A$. Clearly $\alpha$ is either identity or $\alpha: c\mapsto c^{-1}$, which implies that $G$
is isomorphic to the fundamental group of a $2$-torus or Klein bottle. Hence part (i) of the lemma. 

As for part (ii), suppose to the contrary, there is a pair of hyperbolic-hyperbolic elementary 
Z-splittings $G=A_i\ast_{C_i} B_i$ (or $A_i\ast_{C_i}$), $i=1,2$, where $C_i$ is generated by $c_i$. 
We first note that the hyperbolicity implies that the splittings are reduced. Then by part (i), there are 
integers $m,n>0$ such that $c_1^m, c_2^n\in z(G)$, so that $c_1^m$ and $c_2^n$
are commutative. With this understood, 
Theorem 3.6 in Rips-Sela \cite{RipS} implies that $G$ is isomorphic to the fundamental group 
of either a $2$-torus, or a Klein bottle, or an Euclidean $2$-branched projective plane, or an Euclidean
$4$-branched sphere (an explicit presentation of these groups are given in Proposition 3.3 of
\cite{RipS}, p. 63). The case of $2$-torus or Klein bottle is excluded by the assumptions of the lemma,
and the rest of the cases are excluded by the fact that $G$ has infinite center (see Lemma 2.2). 
(Note that in Theorem 3.6 of \cite{RipS}, there is the assumption that $G$ is a freely indecomposable 
group which does not split over $\Z_2$. By Stallings End Theorem, cf. e.g., \cite{SW}, Theorem 6.1,
$G$ satisfies this assumption because of being single-ended.) Hence the lemma.

\end{proof}

We remark that hyperbolic-hyperbolic splittings do occur. For example, let $G$ be the fundamental 
group of a Klein bottle. Then $G=A\ast_A\alpha$, where $A=\langle c\rangle$ is infinite cyclic and 
$\alpha: c\mapsto c^{-1}$, and $G=A\ast_C A$, where $C$ is the index $2$ subgroup of the infinite
cyclic group $A$, are a pair of hyperbolic-hyperbolic splittings of $G$. 

Let $G$ be a single-ended group with infinite $z(G)$, which is not isomorphic to the fundamental
group of a $2$-torus or Klein bottle. Let $T$ be the Bass-Serre tree of a reduced Z-splitting of $G$,
and let $V$ be the subset of the set of vertices $VT$ which consists of $v$ such that the isotropy 
subgroup $G_v$ fixes a vertex $v^\prime\neq v$. The subset $V$ is clearly $G$-invariant, which
gives rise to a $G$-invariant partition $(V, VT\setminus V)$ of $VT$. The following lemma is concerned with the structure of $V$.

\begin{lemma}
There exists a collection of infinite cyclic subgroups $G_i$ of $G$, $i\in I$,  which has the following significance.
\begin{itemize}
\item For each $i\in I$, let $V_i$ be the subset of $V$ consisting of $v$ such that $G_v=G_i$, 
and let $H_i\equiv \{t\in G|tgt^{-1}=g, \forall g\in G_i\}$ be the centralizer of $G_i$. Then 
$H_i$ acts transitively on $V_i$.
\item For each $i\in I$, let $\{g_j|j\in J(i)\}$ be a fixed choice of representatives of the right cosets
of $H_i$ in $G$, where the right coset $H_i$ is represented by $g_j=1$. Then 
$\{g_j(V_i)| j\in J(i), i\in I\}$ forms a partition of $V$. 
\end{itemize}
\end{lemma}

\begin{proof}
Let $v\in V$ be any element, and let $v^\prime\neq v$ be fixed under $G_v$. Since $T$ is a tree,
there exists a unique reduced path $\gamma$ in $T$ which connects $v$ and $v^\prime$. 
Because $G_v$ fixes both $v$ and $v^\prime$, and because $\gamma$ is unique, $G_v$ must
also fix $\gamma$. In particular, if $e$ is the edge in $\gamma$ which is incident to $v$, then
it follows easily that $G_v=G_e$, which implies that $G_v$ is infinite cyclic. 

Let $v_1$ be the other vertex in $\gamma$ to which $e$ is incident. Since the Z-splitting is 
reduced, $v_1$ must lie in the same orbit of $v$ under the action of $G$. In other words, there
is a $t\in G$ such that $t\cdot v=v_1$. Suppose $G_v=G_e$ is generated by $c$. Then
$G_{v_1}=tG_v t^{-1}$ is generated by $c_1\equiv tct^{-1}$. Furthermore, $c\in G_e\subset
G_{v_1}$, so that $c=c_1^n$ for some $n\in \Z$. On the other hand, by Lemma 3.1(i), there exists 
a nonzero $m\in \Z$ such that $c^m\in z(G)$. Consequently, 
$$
c_1^m=(tct^{-1})^m=tc^m t^{-1}=c^m=c_1^{nm},
$$
which implies $n=1$. With $c=c_1=tct^{-1}$, it follows that $t$ lies in the centralizer
of $G_v$, and moreover, $G_{v_1}=G_v$. Repeating this argument to $v_1$, we see that
there is a $t^\prime$ lying in the centralizer of $G_v$, such that $t^\prime \cdot v=v^\prime$
and $G_{v^\prime}=G_v$. Now if we let $V(v)$ be the subset of $V$ consisting of elements 
whose isotropy subgroup equals $G_v$, and let $H(v)$ be the centralizer of $G_v$, then
$H(v)$ acts transitively on $V(v)$. 

The above analysis shows that the following relation $\sim$ on $V$ is an equivalence relation:
$v^\prime\sim v$ if and only if $G_v$ fixes $v^\prime$. The equivalence relation gives rise to
a partition of $V$. It is clear that one can choose a subset $\{V_i|i\in I\}$ of equivalence classes
such that this partition can be described as $\{g_j(V_i)| j\in J(i), i\in I\}$, where $G_i$ is the
isotropy subgroup of the vertices in $V_i$, and $g_j$, $j\in J(i)$, is some fixed representative
of the right coset of the centralizer $H_i$ of $G_i$ in $G$, with $g_j=1$ for the right coset of
$H_i$.  This completes the proof of the lemma. 

\end{proof}

\subsection{Proof of Theorem 1.5(1)}
By assumption $X$ is fiber-sum-decomposed into $X_i$ along $N_j$. This gives rise to a Z-splitting
of $\pi_1(X)$ which will be denoted by $\Lambda$, with vertex groups and edge groups given
by $\pi_1(X_i)$ and $\pi_1(N_j)$ respectively. Note that Definition 1.3(iv) implies that the
Z-splitting $\Lambda$ is reduced. Furthermore, we shall point out that by Lemma 3.1(i), 
$z(\pi_1(X))$ is contained in every edge group of $\Lambda$. On the other hand, recall that the 
fiber-sum decomposition of $X$ gives rise to a canonical injective $\s^1$-action on $X$. We
denote the orbit map by $\pi:X\rightarrow Y$, where we shall point out that $Y$ is naturally a good 
orbifold, i.e., it does not contain any bad $2$-suborbifolds. Let $\Sigma_j$ be the spherical 
$2$-suborbifold of $Y$ over which $N_j$ is Seifert fibered under $\pi$. Then it follows easily that the decomposition of $Y$ in $Y_i$ along $\Sigma_j$ is a reduced spherical decomposition, where 
$Y_i$ is the irreducible $3$-orbifold in the orbit map $\pi_i: X_i\rightarrow Y_i$ that comes with the
fiber-sum decomposition of $X$ (cf. Definition 1.3). 

Let $\Lambda_{JSJ}$ be a canonical JSJ decomposition of $\pi_1(X)$ as constructed 
in \cite{RipS}. We will show that $\Lambda_{JSJ}$ and
$\Lambda$ are equivalent as canonical JSJ decompositions of $\pi_1(X)$
as described in \cite{RipS}. To this end, 
we consider the Bass-Serre trees $T_{JSJ}$, $T$ of $\Lambda_{JSJ}$, $\Lambda$, each equipped
with the canonical action of $\pi_1(X)$. As for notations, recall that for any vertex $v$ or edge 
$e$ of $T_{JSJ}$ or $T$, the corresponding isotropy subgroups of $\pi_1(X)$ are denoted by 
$G_v$ and $G_e$ respectively. 

\begin{lemma}
For any $w\in VT$, $G_w$ fixes a vertex of $T_{JSJ}$. 
\end{lemma}

\begin{proof}
We consider the induced action of $G_w$ on the Bass-Serre tree $T_{JSJ}$, and for any vertex $v$ 
and edge $e$ of $T_{JSJ}$, we denote by $G_v^\prime$, $G_e^\prime$ the isotropy subgroups of
the $G_w$-action at $v$ and $e$ respectively. By Theorem 4.12 in \cite{DD}, there are following
three possibilities.

\begin{itemize}
\item [{(a)}] $G_w$ fixes a vertex of $T_{JSJ}$.
\item [{(b)}] There is a reduced infinite path $v_0,e_1^{\epsilon_1},v_1, e_2^{\epsilon_2}, \cdots,$ in
$T_{JSJ}$ such that 
$$
G_{v_0}^\prime\subset G_{v_1}^\prime\subset \cdots,\;\; 
G_w=\bigcup_{n\geq 0} G_{v_n}^\prime=\bigcup_{n\geq 1}G_{e_n}^\prime,
$$
and for all $n\geq 1$, $G_w\neq G_{e_n}^\prime$.
\item [{(c)}] Some element of $G_w$ translates some edge $e$ of $T_{JSJ}$, and for $C\equiv G_e^\prime$,
either $G_w=B\ast_C D$ with $B\neq C\neq D$, or $G_w=B\ast_C$.
\end{itemize}

It remains to show that neither (b) nor (c) can occur. First, applying Lemma 3.1(i) to the 
$\pi_1(X)$-action on $T_{JSJ}$, we see that $z(\pi_1(X))$ fixes every edge of $T_{JSJ}$. 
Secondly, note that there is a factor $X_i$ such that $G_w$ is conjugate to the subgroup
$\pi_1(X_i)$ in $\pi_1(X)$. Finally, if $h$ denotes the homotopy class of a regular fiber of 
$\pi:X\rightarrow Y$, then $h\in z(\pi_1(X))\cap G_w$, so that $h\in G^\prime_e$ for every 
edge $e$ of $T_{JSJ}$. 

With the preceding understood, we consider case (b) first. In this case, we have 
$$
\pi_1^{orb}(Y_i)\cong\pi_1(X_i)/\langle h\rangle\cong G_w/\langle h\rangle 
= \bigcup_{n\geq 1}G_{e_n}^\prime/\langle h\rangle=\bigcup_{n\geq 1} F_n,
$$
where $F_n$ is a finite group, $F_n\subset F_{n+1}$, and $G_w/\langle h\rangle \neq F_n$ for all 
$n\geq 1$. Clearly, $\pi_1^{orb}(Y_i)$ can not be finite. To rule out the case where $\pi_1^{orb}(Y_i)$ 
is infinite, we note that $\pi_1^{orb}(Y_i)$ has a finite index torsion-free subgroup $H$ by the 
Geometrization Theorem (cf. \cite{BLP, MM}). Let $\tilde{H}$ be the corresponding subgroup of 
$G_w/\langle h\rangle$ under $\pi_1^{orb}(Y_i)\cong G_w/\langle h\rangle$. Then 
$\tilde{H}=\bigcup_{n\geq 1} F_n\cap \tilde{H}=\bigcup_{n\geq 1} \emptyset=\emptyset$, which is a contradiction. Hence case (b) is excluded.

For case (c), we set $C^\prime=C/\langle h\rangle$, $B^\prime=B/\langle h\rangle$,
and $D^\prime=D/\langle h\rangle$, then 
$$
G_w/\langle h\rangle = B^\prime\ast_{C^\prime} D^\prime, \mbox{ with } B^\prime\neq C^\prime\neq D^\prime,
\mbox { or } G_w/\langle h\rangle =B^\prime\ast_{C^\prime}.
$$
Since $C^\prime$ is a finite group, $G_w/\langle h\rangle$ has more than one ends by Stallings End Theorem (cf. e.g., \cite{SW}, Theorem 6.1).  However, since $Y_i$ is irreducible, the number of ends of 
$\pi_1^{orb}(Y_i)$ is at most $1$, which is a contradiction to $\pi_1^{orb}(Y_i)\cong G_w/\langle h\rangle$.
This rules out case (c), and the lemma is proved. 

\end{proof}

\begin{lemma}
There exists a $\pi_1(X)$-equivariant bijection $\phi: VT\rightarrow VT_{JSJ}$. In particular, 
for any $w\in VT$, $G_w=G_{\phi(w)}$. 
\end{lemma}

\begin{proof}
First, we let $V$ (resp. $V_{JSJ}$) be the subset of $VT$ (resp. $VT_{JSJ}$) described in 
Lemma 3.2, and let $G_i$, $V_i$, $H_i$, $g_j$, $j\in J(i)$, $i\in I$, be as defined in 
Lemma 3.2 for $VT$. 

Given any $w\in VT$, $G_w$ fixes a vertex $v\in VT_{JSJ}$ by Lemma 3.3. On the other hand,
since $\pi_1(X)$ has no hyperbolic-hyperbolic splittings (cf. Lemma 3.1(ii)), it follows from the construction of canonical JSJ decompositions in \cite{RipS} that the action of $G_v$ on $T$ 
must also fix a vertex, say $w^\prime$. One has the obvious inclusion relations
$G_w\subset G_v\subset G_{w^\prime}$. By Lemma 3.2, one always has $G_w=G_{w^\prime}$,
so that $G_v=G_w$ must hold. We will discuss according to (i) $w\in VT\setminus V$, (ii) $w\in V$.

In case (i), $w^\prime=w$. We claim that $v\in VT_{JSJ}\setminus V_{JSJ}$, in particular, $v$ is uniquely determined by $w$. To see this, suppose there is a $v_1\neq v$ such that 
$G_{v_1}=G_v$. Then by Lemma 3.2
there is a $t$ lying in the centralizer of $G_v$ such that $v_1=t\cdot v$. In particular, $t$ is not
in $G_v=G_w$.  This implies that $t\cdot w\neq w$, but $G_{t\cdot w}=G_w$, which is a contradiction
to the assumption that $w\in VT\setminus V$.  With this understood, we define $\phi$ from 
$VT\setminus V$ to $VT_{JSJ}\setminus V_{JSJ}$ by setting $\phi(w)=v$.
It follows easily that $\phi$ is a $\pi_1(X)$-equivariant bijection
between $TV\setminus V$ and $VT_{JSJ}\setminus V_{JSJ}$. (The surjectivity part uses the fact
that for any vertex $v\in VT_{JSJ}$, the action of $G_v$ on $T$ fixes a vertex. This is a consequence
of Lemma 3.1(ii) by the construction of JSJ decompositions in \cite{RipS}.)

In case (ii) where $w\in V$, $v$ also lies in $V_{JSJ}$ by a similar argument. We shall define 
$\phi:V\rightarrow V_{JSJ}$ as follows. Let $V_{i,JSJ}$ be the subset of $V_{JSJ}$ consisting of
vertices whose isotropy subgroups are given by $G_i$. Then for any fixed choice of $w_i\in V_i$,
$v_i\in V_{i,JSJ}$, there is a $H_i$-equivariant bijection $\phi:V_i\rightarrow V_{i,JSJ}$ sending
$w_i$ to $v_i$. Using the elements $g_j$, $j\in J(i)$, we can uniquely extend $\phi$ to a
$\pi_1(X)$-equivariant bijection from $\bigcup_{j\in J(i)} g_j(V_i)$ to 
$\bigcup_{j\in J(i)} g_j(V_{i,JSJ})$, which defines $\phi$ from $V$ to $V_{JSJ}$. This completes
the proof of the lemma. 

\end{proof}

According to Rips-Sela \cite{RipS}, Theorem 7.1, canonical JSJ decompositions of a single-ended, 
finitely presented group $G$ are determined up to the following equivalence relation: the
Bass-Serre trees are $G$-homotopy equivalent relative to the set of vertices. With this 
understood, Theorem 1.5(1) follows from part (1) of the following proposition. In (2)-(4) we 
list some consequences of (1) which will be used later in the proofs of Theorem 1.5(2), 
Theorem 1.1, and Theorem 1.2. 

\begin{proposition}
(1) There exist subdivisions $T^\prime$, $T^\prime_{JSJ}$ of $T$, $T_{JSJ}$ respectively, and
$\pi_1(X)$-equivariant simplicial maps $h_1: T^\prime\rightarrow T_{JSJ}$, 
$h_2: T^\prime_{JSJ}\rightarrow T$ extending $\phi$ and $\phi^{-1}$ ($\phi$ as in Lemma 3.4),
such that $h_2\circ h_1$ and $h_1\circ h_2$ are $\pi_1(X)$-homotopic, relative to the set of vertices, 
to the corresponding identity maps. 

(2) There exists a bijection $\hat{\phi}: V\Lambda\rightarrow V\Lambda_{JSJ}$, such that for any
factor $X_i$ of the fiber-sum decomposition of $X$, $\pi_1(X_i)$ is conjugate 
in $\pi_1(X)$ to the vertex group at the vertex $\hat{\phi}(X_i)$ of $\Lambda_{JSJ}$. 
In particular, the number of factors $X_i$ and the conjugacy classes of subgroups 
$\pi_1(X_i)$ depend only on $\pi_1(X)$.

(3) The cardinality of $\{N_j\}$ depends only on $\pi_1(X)$. 

(4) For any $N_j$, there is an edge $e_j$ of the graph of $\Lambda_{JSJ}$ such that 
$\pi_1(N_j)$ is conjugate in $\pi_1(X)$ to the edge group at $e_j$, and vice versa. In 
particular, the set of conjugacy classes
of subgroups $\pi_1(N_j)$ depends only on $\pi_1(X)$. 
\end{proposition}

\begin{proof}
Fixing a choice of $\phi$ in Lemma 3.4, we shall define the subdivision $T^\prime$ of $T$ and the
simplicial map $h_1: T^\prime\rightarrow T$ as follows. For any edge $e\in ET$, there is a unique
reduced path in $T_{JSJ}$ which starts from $\phi(\iota e)$ and ends at $\phi(\tau e)$. There is a
unique subdivision of $e$ such that $\phi$ can be extended to a simplicial map over $e$. Doing this 
to every edge of $T$, we obtained the subdivision $T^\prime$ and the simplicial map $h_1$. The whole construction is clearly $\pi_1(X)$-equivariant because $\phi$ is $\pi_1(X)$-equivariant and 
reduced paths with fixed ends in a tree are unique. The subdivision $T^\prime_{JSJ}$ and the simplicial map $h_2$ are constructed similarly with $\phi$ replaced by $\phi^{-1}$. One can further subdivide $T^\prime$ (still denoted by $T^\prime$ for simplicity) so that $h_1$ can be regarded 
as a simplicial map to the subdivision $T^\prime_{JSJ}$ of $T_{JSJ}$. With this understood, 
$h_2\circ h_1: T^\prime\rightarrow T$ is $\pi_1(X)$-homotopic to the identity map 
relative to the set of vertices $VT$ because (i) it is identity on $VT$, 
and (ii) $T$ is a tree. The statement about $h_1\circ h_2$ follows similarly. 
This finishes the proof of part (1). 

Part (2) is a direct consequence of Lemma 3.4.  For part (3), recall that the set of edges of 
$\Lambda$ is identified with the set $\{N_j\}$. With this understood, observe that the underlying graphs of $\Lambda$ and $\Lambda_{JSJ}$, which are given by $T/\pi_1(X)$ and 
$T_{JSJ}/\pi_1(X)$ respectively, are homotopy equivalent, so that they have the same Euler characteristics. This shows that the Euler characteristic of $\Lambda$, i.e., the number of vertices minus the number of edges of $\Lambda$, depends only on $\pi_1(X)$. It follows that the 
cardinality of $\{N_j\}$ depends only on $\pi_1(X)$.

Finally, we give a proof for part (4). For any $N_j$, we choose an edge $e$ of $T$ whose
$\pi_1(X)$-orbit corresponds to $N_j$. As we have shown in the proof of part (1), 
$h_2\circ h_1(e)$ is a path in $T$ which has the same initial and terminal points as $e$.
Since $T$ is a tree, the loop formed by $h_2\circ h_1(e)$ and $e^{-1}$ must be reduced,
which implies that $e$ lies in the image of $h_2\circ h_1(e)$. Let $e^\prime$ be an edge
of $T_{JSJ}$ lying in the path $h_1(e)$ such that $e$ is contained in the path $h_2(e^\prime)$.
Then by the construction of $h_1,h_2$ in part (1), we have 
$G_e\subset G_{e^\prime}\subset G_e$,
which implies that $G_e=G_{e^\prime}$. We name $e_j$ to be the edge of $\Lambda_{JSJ}$ 
which corresponds to the $\pi_1(X)$-orbit of $e^\prime$. Then it follows that $\pi_1(N_j)$ is 
conjugate to the edge group of $\Lambda_{JSJ}$ at $e_j$. Part (4) follows easily. 
This completes the proof of Proposition 3.5.

\end{proof}

\subsection{Proof of Theorem 1.5(2)}
Before turning to the proof of Theorem 1.5(2), we first give a geometric interpretation 
of the conjugacy classes of subgroups $\pi_1(N_j)$ in $\pi_1(X)$. We begin by observing that the submanifolds $N_j$ fall into two different types as follows. Let $\Gamma$ be the subgroup of $\pi_1(X)$ generated by the homotopy class of a regular fiber of $\pi: X\rightarrow Y$. Then $N_j$ falls into two cases 
according to (i) $\Gamma=\pi_1(N_j)$, or (ii) $\Gamma$ is a proper subgroup of $\pi_1(N_j)$.
It is clear that case (i) corresponds to the case where $\Sigma_j$ is an ordinary $2$-sphere. 

With the preceding understood, we have

\begin{lemma}
(1) Suppose $\Gamma$ is a proper subgroup of $\pi_1(N_j)$ for some $j$. Then for any $N_k$, if $g^{-1}\pi_1(N_j)g\subset \pi_1(N_k)$ for some $g\in \pi_1(X)$, then $g^{-1}\pi_1(N_j)g=\pi_1(N_k)$.
In particular, if $\pi_1(N_j)=z(\pi_1(X))$, then $\pi_1(N_k)=z(\pi_1(X))$ for any $k$.

(2) Let $N_j,N_k$ be given which are over $\Sigma_j$, $\Sigma_k$ respectively. Suppose
there are components $\gamma_j$, $\gamma_k$ of the singular set of $Y$ such that
$\Sigma_j\cap \gamma_j\neq \emptyset$, $\Sigma_k\cap \gamma_k\neq \emptyset$, and
suppose that $\pi_1(N_j)$, $\pi_1(N_k)$ are conjugate in $\pi_1(X)$. Then $\gamma_j$,
$\gamma_k$ are equivalent in the following sense: either $\gamma_j=\gamma_k$, or
there are components of the singular set of $Y$, $\gamma_0, \gamma_1, \cdots, \gamma_N$, 
and spherical $2$-suborbifolds $\Sigma_1,\cdots,\Sigma_N\in \{\Sigma_j\}$, such that
$$
\gamma_{\alpha-1} \cap \Sigma_\alpha\cap \gamma_\alpha\neq \emptyset, \;\;
\alpha=1,2,\cdots, N.
$$
\end{lemma}

\begin{proof}
For part (1), let $N_j,N_k$ be Seifert fibered over $\Sigma_j$, $\Sigma_k$ under
$\pi:X\rightarrow Y$. Since $\Gamma$
is a proper subgroup of $\pi_1(N_j)$ and $g^{-1}\pi_1(N_j)g\subset \pi_1(N_k)$ for some $g\in \pi_1(X)$, 
$\Gamma$ is also a proper subgroup of $\pi_1(N_k)$. Consequently, there are
components $\gamma_j$, $\gamma_k$ of the singular set of $Y$, such that 
$\Sigma_j\cap \gamma_j\neq \emptyset$, $\Sigma_k\cap \gamma_k\neq \emptyset$.
If $\gamma_j=\gamma_k$, one clearly has $g^{-1}\pi_1(N_j)g=\pi_1(N_k)$ as claimed.

Suppose $\gamma_j\neq\gamma_k$. We denote by $Y_0$ the $3$-orbifold obtained from $Y$ by
removing a regular neighborhood of all singular circles of $Y$ except $\gamma_k$. Note that $Y_0$
is a good $3$-orbifold as $Y$ is good. We let $\hat{Y}_0$ be a $3$-manifold cover of $Y_0$. 
We shall apply the Equivariant Loop Theorem (cf. e.g. \cite{BMP}, Theorem 3.19) to $\hat{Y}_0$
as follows. Denote by $F$ a component of $\partial \hat{Y}_0$ which contains the pre-image of
a meridian of $\gamma_j$. Then observe that the assumption $g^{-1}\pi_1(N_j)g\subset \pi_1(N_k)$ for some $g\in \pi_1(X)$ implies that $F$ is not $\pi_1$-injective. Hence by the Equivariant Loop Theorem,
there is an equivariant compression $2$-disc $\hat{D}$ in $\hat{Y}_0$ with $\partial \hat{D}\subset F$. 
The group action on $\hat{D}$ contains exactly one fixed point, which implies that the image of $\hat{D}$
under the covering map $\hat{Y}_0\rightarrow Y_0$ is an embedded $2$-disc $D$ in $|Y_0|$ intersecting
$\gamma_k$ at exactly one point. Furthermore, it follows easily that $\partial D$ must be a meridian
of $\gamma_j$. Closing up $D$ in $|Y|$, we obtain an embedded $2$-sphere $\Sigma$, which intersects each of $\gamma_j$, $\gamma_k$ at exactly one point and intersects no other singular circles. 
Since $Y$ contains no bad $2$-suborbifolds, it follows that $\gamma_j$, $\gamma_k$ must have the
same multiplicity, which implies that $g^{-1}\pi_1(N_j)g=\pi_1(N_k)$ as claimed. If 
$\pi_1(N_j)=z(\pi_1(X))$, then $\pi_1(N_j)\subset \pi_1(N_k)$ for any $k$ by Lemma 3.1(i),
which implies that $\pi_1(N_k)=z(\pi_1(X))$ for any $k$. This finishes off the proof of part (1). 

Next we prove part (2). The idea is to show that up to replacing one or both of $\gamma_j$,
$\gamma_k$ by some singular circles that are equivalent in the sense described in 
part (2) of the lemma, the embedded $2$-sphere $\Sigma$ which we constructed in the previous 
paragraph can be modified so that it lies in
the complement of the spherical $2$-suborbifolds $\{\Sigma_j\}$. To this end, we first perturb
$\Sigma$ so that it intersects each element of $\{\Sigma_j\}$ transversely and the intersection
occurs in the complement of the singular set of $Y$. Now we fix our attention on a 
$\Sigma^\prime\in\{\Sigma_j\}$ such that $\Sigma\cap \Sigma^\prime\neq \emptyset$. Let 
$l\in \Sigma\cap\Sigma^\prime$ be a circle (if there is any) which bounds a disc 
$D\subset \Sigma^\prime$ such that (i) $D$ contains no singular points, (ii) $D$ contains no intersection points with $\Sigma$. Let $D_1,D_2$ be the two discs into which $l$ divides 
$\Sigma$. Then both $D_1\cup D$, $D_2\cup D$ are embedded $2$-spheres in $Y$. Since $Y$
contains no bad $2$-suborbifolds, it follows easily that exactly one of $D_1$ and $D_2$, 
say $D_1$, contains no singular points. With this understood, we shall modify $\Sigma$
by replacing $D_1$ with $D$ and slightly perturbing it by an isotopy so that the new surface 
does not intersect $\Sigma^\prime$ in a neighborhood of $D$. In order to keep the notation simple, we shall still denote the resulting embedded $2$-sphere by $\Sigma$. It is easily seen that the 
above procedure has the effect of removing the component $l$ from $\Sigma\cap \Sigma^\prime$, and moreover,  it does not create new intersection points of $\Sigma$ with any element of 
$\{\Sigma_j\}$.  By repeating this procedure, we may assume now that the intersection of
$\Sigma$ with any element $\Sigma^\prime\in \{\Sigma_j\}$ is either empty, or consists of a 
union of circles each of which divides $\Sigma^\prime$ into two discs, each containing exactly
one singular point. 

One can further reduce the number of components of $\Sigma\cap\Sigma^\prime$ to at most
one. To see this, let $l,l^\prime$ be a pair of components of $\Sigma\cap\Sigma^\prime$ 
such that $l,l^\prime$ bounds an annulus $A^\prime\subset\Sigma^\prime$ and $l$ bounds 
a disc $D^\prime\subset \Sigma^\prime$ where $A^\prime, D^\prime$ do not contain any 
intersection points with $\Sigma$ (note that if the number of components of 
$\Sigma\cap\Sigma^\prime$ is greater than $1$, such a pair always exists). Then the annulus
$A\subset\Sigma$ bounded by $l,l^\prime$ does not contain any singular points, because 
otherwise, either $l$ or $l^\prime$, say $l$, will bound a disc $D\subset \Sigma$ containing no singular points, and furthermore, $D$ and a disc in $\Sigma^\prime$ bounded by $l$ form
an embedded $2$-sphere in $Y$ containing exactly one singular point, contradicting
the fact that $Y$ is pseudo-good. With this understood, we modify $\Sigma$ by replacing
the annulus $A$ with $A^\prime$, and as before, after applying a small isotopy the pair of
components $l,l^\prime$ are removed and no new intersection points are created. By repeating
this procedure, we may assume that for each $\Sigma^\prime\in\{\Sigma_j\}$, 
the intersection $\Sigma\cap\Sigma^\prime$ consists of at most one component. 

Now we are at the final stage of modifying $\Sigma$. Let $l$ be a circle of intersection of
$\Sigma$ with a $\Sigma^\prime\in \{\Sigma_j\}$ such that $l$ bounds a disc $D\subset \Sigma$ 
which does not intersect with any other elements of $\{\Sigma_j\}$. (Such $l$ always exists, or
$\Sigma$ lies in the complement of $\{\Sigma_j\}$.) Let $D^\prime\subset \Sigma^\prime$ be
a disc bounded by $l$. Then $D\cup D^\prime$ is an embedded $2$-sphere which can be
perturbed so that it lies in the complement of $\{\Sigma_j\}$. Call it $\hat{\Sigma}$, and suppose 
that $\hat{\Sigma}$ lies in $Y_i$, which is an irreducible $3$-orbifold. Furthermore, without loss of 
generality we assume $D$ contains a singular point in $\gamma_j$, and we denote by 
$\gamma_j^\prime$ the singular circle which intersects with $D^\prime$. We claim
that $\gamma_j$, $\gamma_j^\prime$ are equivalent in the sense described in part (2) of
the lemma. To see this, note that $\hat{\Sigma}$ bounds a discal $3$-orbifold in $Y_i$ by
the irreducibility of $Y_i$. In particular, there is an arc $\gamma$ lying in the singular set
of $Y_i$ which connects the two singular points on $\hat{\Sigma}$. If $\gamma$ does not 
intersect any elements of $\{\Sigma_j\}$, then $\gamma_j=\gamma_j^\prime$, hence
are equivalent. Suppose $\Sigma_1, \cdots, \Sigma_N$ are the elements of $\{\Sigma_j\}$ 
which intersect with $\gamma$. Then there are sub-arcs $I_1,\cdots,I_N$ of $\gamma$, where
$I_\alpha$ is contained in the discal $3$-orbifold in $Y_i$ bounded by $\Sigma_\alpha$,
$1\leq \alpha\leq N$. Clearly there are singular circles 
$\gamma_0,\gamma_1,\cdots,\gamma_N$ such that the end points of $I_\alpha$ lie in 
$\gamma_{\alpha-1}$, $\gamma_\alpha$ respectively.  It follows easily that $\gamma_j$, 
$\gamma_j^\prime$ are equivalent through $\gamma_0,\cdots,\gamma_N$ and $\Sigma_1,
\cdots,\Sigma_N$. With this understood, we replace $\gamma_j$ by $\gamma_j^\prime$,
and we modify $\Sigma$ by replacing $D$ by $D^\prime$. The new embedded $2$-sphere
can be perturbed slightly so that it does not intersect $\Sigma^\prime$ and no new intersection
points with elements of $\{\Sigma_j\}$ were created. Furthermore, it intersects with each of the singular circles $\gamma_k,\gamma_j^\prime$ in exactly one point and contains no
other singular points. By repeating this procedure, we obtain an embedded $2$-sphere,
which is still denoted by $\Sigma$, such that (i) $\Sigma$ is in the complement of the elements
of $\{\Sigma_j\}$, and (ii) $\Sigma$ contains exactly two singular points lying on some singular
components $\hat{\gamma}_j,\hat{\gamma}_k$, which are equivalent to $\gamma_j$,
$\gamma_k$ respectively. As we have shown earlier, $\hat{\gamma}_j,\hat{\gamma}_k$ are
equivalent, which implies that $\gamma_j$, $\gamma_k$ are equivalent. 
This finishes the proof of the lemma.

\end{proof}

In summary, the conjugacy classes of subgroups $\pi_1(N_j)$ (which are the conjugacy classes
of the edge groups of $\Lambda$) can be classified as follows: (i) there is a distinguished 
conjugacy class, i.e., the class of those $\pi_1(N_j)=\Gamma$; this conjugacy class can be 
characterized by the fact that the corresponding $\Sigma_j$ are ordinary $2$-spheres; (ii) for
any other conjugacy class where $\pi_1(N_j)$ contains $\Gamma$ as a proper subgroup, there
is an associated equivalence class of singular circles as described in Lemma 3.6(2), which
is characterized by the fact that $\pi_1(N_j)$ belongs to the conjugacy class if and only if the
corresponding $\Sigma_j$ intersects with a singular circle belonging to the equivalence
class. With this understood, we shall show in the next lemma that by modifying the embeddings
of $N_j$ via fiber-preserving isotopies (with respect to $\pi:X\rightarrow Y$) if necessary, one can bring the underlying graph of the Z-splitting $\Lambda$ into a certain normal form. We should 
point out that modifying the embeddings of $N_j$ via fiber-preserving isotopies does not 
change the conjugacy classes of the edge groups of the Z-splitting.

\begin{lemma}
For any given vertex $v$ of $\Lambda$, and any conjugacy class of edge groups of $\Lambda$
which are contained in the vertex group $G(v)$ up to conjugacy, one can modify
the embeddings of those $N_j$ via fiber-preserving isotopies, where $\pi_1(N_j)$ belongs to the given conjugacy class of edge groups, such that the Z-splitting of $\pi_1(X)$ associated to the
new fiber-sum decomposition of $X$ has the following property: for any edge $e$, if $G(e)$ 
belongs to the given conjugacy class of edge groups, then $e$ is incident to $v$. 
\end{lemma}

\begin{proof}
First of all, we observe that modifying the embeddings of $N_j$ via fiber-preserving isotopies corresponds to moving one of the points $y_{j,1},y_{j,2}$ (cf. Definition 1.3) via isotopies, and moreover, for any $Y_i$, the edge which corresponds to $N_j$ is incident to the vertex corresponding to $X_i$ if and only if one of the points $y_{j,1},y_{j,2}$ lies in $Y_i$. 

Now with the vertex $v$ and the conjugacy class of edge groups given as in the lemma, 
we denote by $X_0$ the irreducible $\s^1$-four-manifold corresponding to $v$, 
and denote by $Y_0$ the corresponding irreducible $3$-orbifold. We first note that the case 
where the given conjugacy class of edge groups is the distinguished one, i.e., where 
$\pi_1(N_j)=\Gamma$, is trivial, because in this case
$\Sigma_j$ is an ordinary $2$-sphere and hence the points $y_{j,1},y_{j,2}$ are both lying in the complement of the singular set. For any other conjugacy class of edge groups, there is an 
associated equivalence class of singular circles as described in Lemma 3.6(2). 
Since the edge groups belonging to the given conjugacy class are contained in the vertex group $G(v)=\pi_1(X_0)$ up to conjugacy, 
there must be a singular circle belonging to the equivalence class which has nonempty intersection with the irreducible $3$-orbifold $Y_0$. We pick one such singular circle and denote it by $\gamma_0$, and we set
$I_0\equiv Y_0\cap \gamma_0\neq \emptyset$. Now consider any $N_j$ such that $\pi_1(N_j)$ belongs to the given conjugacy class of edge groups and $\Sigma_j\cap \gamma_0\neq
\emptyset$. There are two possibilities: (i) $\Sigma_j$ intersects $\gamma_0$ at two points;
(ii) $\Sigma_j$ intersects $\gamma_0$ at only one point. Consider case (i) first.  If we cut $Y$ 
open along $\Sigma_j$ and then fill in the $3$-discal neighborhood of $y_{j,1},y_{j,2}$, the 
singular circle $\gamma_0$ is turned into two components, one of which, denoted by 
$\gamma^\prime$, contains $I_0$. Without loss of generality, assume $y_{j,1}$ is contained 
in $\gamma^\prime$. Then by moving $y_{j,1}$ along $\gamma^\prime$ via isotopy if 
necessary, we may arrange such that $y_{j,1}\in I_0$. Now consider case (ii).
Let $\gamma_1$ be the singular circle which contains the other singular point on 
$\Sigma_j$. Then when we cut $Y$ open along $\Sigma_j$ and fill in the $3$-discal 
neighborhoods of $y_{j,1},y_{j,2}$, the two components $\gamma_0$,$\gamma_1$ 
are turned into one component, denoted by $\gamma^\prime$. In this case, one can always 
arrange so that $y_{j,1}\in I_0$, by moving $y_{j,1}$ via isotopy along $\gamma^\prime$. 
Note that after moving $y_{j,1}$ via isotopy and then performing the connected sum operation 
to get back to $Y$, the singular circles $\gamma_0,\gamma_1$ are turned into 
$\gamma_0^\prime$, $\gamma_1^\prime$, both of which have nonempty intersection with 
$Y_0$. With this last property understood, observe that we can now perform the operation 
described above to any $N_j$ such that $\Sigma_j\cap \gamma_1^\prime\neq \emptyset$. 
The lemma follows by an induction process. 

\end{proof}

We remark that applying Lemma 3.7 to a Z-splitting $\Lambda$ does not change the sets 
$V\Lambda$ and $E\Lambda$; it only changes the incident function. From the construction of 
Bass-Serre trees (cf. \cite{DD}), it follows particularly that neither the action of $\pi_1(X)$ on the vertex set of the Bass-Serre tree $T$ of $\Lambda$ changes, nor does the 
$\pi_1(X)$-equivariant bijection $\phi$ in Lemma 3.4. 

\vspace{3mm}

\noindent{\bf Proof of Theorem 1.5(2)}

\vspace{3mm}

First of all, we shall reformulate the problem as follows. We denote the group $\pi_1(X^\prime)$
by $G$ and identify $\pi_1(X)$ with $G$ via the given isomorphism $\alpha: \pi_1(X)
\rightarrow \pi_1(X^\prime)$. With this understood, let $\Lambda,\Lambda^\prime$ be the 
Z-splittings of $G$ associated to the given fiber-sum
decompositions of $X,X^\prime$ respectively. We shall prove that after modifying the 
embeddings of $N_j$, $N_j^\prime$ via fiber-preserving isotopies if necessary, 
$\Lambda,\Lambda^\prime$ may be arranged to be isomorphic as Z-splittings of $G$.
Note that the assumption that $N_j,N_j^\prime$ are null-homologous is equivalent to that 
the underlying graph of $\Lambda,\Lambda^\prime$ is a tree. We shall denote by 
$T,T^\prime$ the Bass-Serre tree of $\Lambda,\Lambda^\prime$ respectively. By Lemma 3.4, 
there exists a $G$-equivariant bijection $\phi$ from $VT$ onto $VT^\prime$, which induces a bijection $\hat{\phi}:V\Lambda\rightarrow V\Lambda^\prime$ and a family of isomorphisms 
of vertex groups $\rho_v: G(v) \rightarrow G({v^\prime})$ given by conjugation 
by elements of $G$, where $v\in V\Lambda$, $v^\prime=\hat{\phi}(v)\in V\Lambda^\prime$.

First consider the special case where $\pi_1(N_j)=z(G)=\pi_1(N_j^\prime)$ for all $N_j,N_j^\prime$. 
We fix a vertex $v\in V\Lambda$ and let $v^\prime=\hat{\phi}(v)\in V\Lambda^\prime$ be the 
corresponding vertex. Then we 
apply Lemma 3.7 to $\Lambda$, $\Lambda^\prime$ so that for the resulting new Z-splittings, 
which are still denoted by $\Lambda$, $\Lambda^\prime$ for simplicity, every 
edge $e\in E\Lambda$, $e^\prime\in E\Lambda^\prime$ is incident to $v,v^\prime$ respectively. 
With this understood, there is
an isomorphism of the underlying graphs of $\Lambda$ and $\Lambda^\prime$, extending 
$\hat{\phi}:V\Lambda\rightarrow V\Lambda^\prime$. Since by assumption all the edge groups of 
$\Lambda$ and $\Lambda^\prime$ are given by the center $z(G)$, it follows easily that the family of
isomorphisms $\rho_v$ can be extended to an isomorphism of the Z-splittings  
$\Lambda$ and $\Lambda^\prime$. This finishes the proof for the special case
where $\pi_1(N_j)=z(G)=\pi_1(N_j^\prime)$ for all $N_j,N_j^\prime$.

Suppose $\pi_1(N_j)=z(G)$ for all $N_j$ does not hold. Then by Lemma 3.6(1), the 
condition that $\pi_1$ of a regular fiber of $\pi: X\rightarrow Y$ is a proper subgroup of $\pi_1(N_j)$ 
for some $N_j$ is equivalent to the more convenient condition that $\pi_1(N_j)\neq z(G)$, 
as the latter is formulated without reference to $\pi: X\rightarrow Y$. On the other hand, 
by Proposition 3.5(4), $\pi_1(N_j^\prime)=z(G)$ for all $N_j^\prime$ 
also does not hold. Accordingly, one can divide the set of edges $E\Lambda$ 
(resp. $E\Lambda^\prime$) into two groups by the following rules:

\begin{itemize}
\item [{(I)}] $e\in E\Lambda$ (resp. $e^\prime\in E\Lambda^\prime$) belongs to (I) if and only if 
$G(e)\neq z(G)$ (resp. $G({e^\prime})\neq z(G)$);
\item [{(II)}] $e\in E\Lambda$ (resp. $e^\prime\in E\Lambda^\prime$) belongs to (II) if and only if 
$G(e)=z(G)$ (resp. $G({e^\prime})=z(G)$).
\end{itemize}

Pick a vertex $v\in V\Lambda$, and without loss of generality, assume that there is an edge $e$
belonging to (I) such that $G(e)$ is conjugate to a subgroup of $G(v)$. We denote the set of such
edges by $E_v$. Then by Lemma 3.7, we can assume that any $e\in E_v$ is incident to $v$. 
Furthermore, we can assume (again with the help of Lemma 3.7) that any $e\in E\Lambda$ belonging to (II) is not incident to $v$ by the fact that $E_v\neq \emptyset$. With this understood, 
we denote by $\Gamma_v$ the minimal subgraph containing $v$ and $E_v$ and by $G_{\Gamma_v}$ the corresponding subgraph of groups supported by $\Gamma_v$. Finally, we 
let $v^\prime=\hat{\phi}(v)\in V\Lambda^\prime$ be the corresponding vertex in the Z-splitting 
$\Lambda^\prime$. We make the same arrangement as above for the vertex $v^\prime$ with the corresponding notations in which $v$ is replaced by $v^\prime$.

Our next goal is to construct an isomorphism between the subgraphs of groups $G_{\Gamma_v}$ and $G_{\Gamma_{v^\prime}}$, extending the given isomorphism $\rho_v: G(v)\rightarrow 
G(v^\prime)$. To this end, we pick a fundamental $G$-transversal for $G_{\Gamma_v}$ as
follows. Let $\tilde{v}$ be a vertex of the Bass-Serre tree $T$ whose $G$-orbit is $v$. For each
$e\in E_v$, we choose an edge $\tilde{e}\in ET$ incident to $\tilde{v}$, whose $G$-orbit is $e$.
We let $\Gamma_{\tilde{v}}$ be the minimal subgraph of $T$ containing $\tilde{v}$ and
$\tilde{e}$, $\forall e\in E_v$. Then it is clear that $\Gamma_{\tilde{v}}$ is a fundamental $G$-transversal for $G_{\Gamma_v}$. With this understood, we shall construct a fundamental $G$-transversal for $G_{\Gamma_{v^\prime}}$ as follows. 

We set $\tilde{v}^\prime=\phi(\tilde{v})$, where $\phi: VT\rightarrow VT^\prime$ is the 
$G$-equivariant bijection coming from Lemma 3.4, which induces 
$\hat{\phi}: V\Lambda\rightarrow V\Lambda^\prime$.
For any edge $\tilde{e}\in \Gamma_{\tilde{v}}$, we denote by
$\tilde{w}$ the vertex other than $\tilde{v}$ to which $\tilde{e}$ is incident to, and set 
$\tilde{w}^\prime=\phi(\tilde{w})$ correspondingly. Then as in the proof of Proposition 3.5,
there exists a unique reduced path in $T^\prime$ connecting $\tilde{v}^\prime$ to $\tilde{w}^\prime$:
$$
v_0=\tilde{v}^\prime, e_1^{\epsilon_1}, v_1, e_2^{\epsilon_2}, \cdots, e_n^{\epsilon_n}, v_n=\tilde{w}^\prime, 
$$
such that $G_{\tilde{e}}\subset G_{e_i}$ for all $i$ and that there exists a $j$ with 
$G_{e_j}=G_{\tilde{e}}$.  Let $\hat{e}_i\in E\Lambda^\prime$ be the $G$-orbit of $e_i$. 
Then since the edge $e\in E_v$ belongs to (I), it follows that $\hat{e}_j\in E\Lambda^\prime$ also belongs to (I) because $G_{e_j}=G_{\tilde{e}}$.
Now with $G_{e_j}=G_{\tilde{e}}\subset G_{e_i}$, it follows from Lemma 3.6(1) that 
$G_{e_j}=G_{e_i}$ for all $i$, which implies that the edge groups $G(\hat{e}_i)$ belong to 
the same conjugacy class in $G$.
It follows that the vertices $v_k$, where $k$ is even, must be in the same $G$-orbit, 
and that $n$ must be odd. 
In particular, $v_{n-1}$ and $v_0=\tilde{v}^\prime$ are in the same $G$-orbit. We fix a choice of 
$g_{\tilde{e}}\in G$ such that $g_{\tilde{e}}v_{n-1}=\tilde{v}^\prime$, and set $\tilde{e}^\prime=e_n$, and
let $w\in V\Lambda$, $w^\prime\in V\Lambda^\prime$ be the $G$-orbit of $\tilde{w}$, $\tilde{w}^\prime$
respectively. Then the $G$-orbit $e^\prime\in E\Lambda^\prime$ of $\tilde{e}^\prime$ is incident to the
vertices $v^\prime$ and $w^\prime$. It follows that $e^\prime, w^\prime$ are part of the subgraph 
$\Gamma_{v^\prime}$, and $v\mapsto v^\prime$, $e\mapsto e^\prime$ and $w\mapsto w^\prime$
define an isomorphism between $\Gamma_v$ and $\Gamma_{v^\prime}$.

Suppose $\rho_v: G(v)\rightarrow G(v^\prime)$ is given by 
$h\mapsto g_{\tilde{v}}h g_{\tilde{v}}^{-1}$ for some $g_{\tilde{v}}\in G$, where $h\in G_{\tilde{v}}$.
Then the subset $\{ g_{\tilde{v}} \tilde{v}^\prime, g_{\tilde{v}} g_{\tilde{e}} \tilde{e}^\prime, g_{\tilde{v}} g_{\tilde{e}} \tilde{w}^\prime|e\in E_v, w\in \Gamma_v\}$ is a fundamental  $G$-transversal for 
$G_{\Gamma_{v^\prime}}$. Moreover, there is an isomorphism 
$\{\rho_v, \rho_e,\rho_w|e\in E_v, w\in \Gamma_v\}$ between the subgraphs of groups
$G_{\Gamma_v}$ and $G_{\Gamma_{v^\prime}}$, extending the given isomorphism $\rho_v: G(v)\rightarrow G(v^\prime)$, where 
$\rho_e: G_{\tilde{e}}\rightarrow G_{g_{\tilde{v}} g_{\tilde{e}} \tilde{e}^\prime}$, 
$\rho_w: G_{\tilde{w}}\rightarrow G_{g_{\tilde{v}} g_{\tilde{e}} \tilde{w}^\prime}$
are given by conjugation of $g_{\tilde{v}} g_{\tilde{e}}\in G$. 

Finally, by repeating the above construction, we obtain a disjoint union of subgraphs of groups 
$G_{\Gamma_k}$ of the Z-splitting $\Lambda$, a disjoint union of subgraphs of groups 
$G_{\Gamma_k^\prime}$ of the Z-splitting $\Lambda^\prime$, and a collection of isomorphisms 
$\rho_k: G_{\Gamma_k}\rightarrow G_{\Gamma_k^\prime}$, such that for any edges $e\in E\Lambda
\setminus \{\Gamma_k\}$, $e^\prime\in E\Lambda^\prime\setminus \{\Gamma_k^\prime\}$, 
$G(e)=z(G)=G(e^\prime)$. It follows easily that the isomorphisms $\rho_k$ can be uniquely extended
to an isomorphism of Z-splittings between $\Lambda$ and $\Lambda^\prime$. 
This finishes the proof of Theorem 1.5(2). 

\section{Irreducible $\s^1$-four-manifolds}

This section is devoted to a proof of Theorem 1.6.  The proof involves a smooth classification of 
fixed-point free, smooth $\s^1$-four-manifolds whose $\pi_1$ has a center of rank greater than $1$ 
(cf. Theorem 4.3), which is given at the end of the section. 

The following lemma shows that a finitely generated group with infinite center is either single-ended or double-ended.

\begin{lemma}
Let $G$ be a finitely generated group with infinite $z(G)$ and suppose $G$ is not single-ended. Then 
$G$ is isomorphic to $A\ast_A \alpha$ where $A$ is a finite group. In particular, $G$ is double-ended.
\end{lemma}

\begin{proof}
Let $e(G)$ denote the number of ends of $G$. Then $e(G)\geq 1$ because $G$ is infinite. 
On the other hand, by Stallings End Theorem
(cf. e.g. Scott-Wall \cite{SW}), if $e(G)\geq 2$, then $G$ splits over a finite subgroup, i.e.,
either $G=A\ast_C B$ with $A\neq C\neq B$, or $G=A\ast_C\alpha$, where in both cases
$C$ is a finite group. By Lemma 2.1, the assumption that $z(G)$ is infinite implies
that the first case can not occur, and in the second case, $C=A=\alpha(C)$. In particular,
$A$ is a finite group.  

\end{proof}

\begin{lemma}
Let $\pi: X\rightarrow Y$ be the orbit map of an injective $\s^1$-action. Then $\pi_1(X)$ is 
double-ended if and only if $\pi_1^{orb}(Y)$ is finite. 
\end{lemma}

\begin{proof}
It suffices to show that if $\pi_1(X)$ is double-ended, then $\pi_1^{orb}(Y)$ is finite; 
the other direction is trivial, cf. e.g. Scott-Wall \cite{SW}. To see this,  note that 
$\pi_1(X)=A\ast_A\alpha$ for a finite group $A$ by Lemma 4.1, where we recall that 
$A\ast_A\alpha$ is generated by elements of $A$ and a letter $t$ with additional relations 
$tat^{-1}=\alpha(a)$, $a\in A$. If we let $H$ be the cyclic subgroup generated by $t$, 
then $H$ has finite index in $\pi_1(X)$. On the other hand, if we let $\Gamma$ be the subgroup 
generated by the homotopy class of a regular fiber of $\pi$, then $\Gamma\cap H$ has 
finite index in $H$ because $\alpha$ is of finite order. Consequently $\Gamma\cap H$ has finite 
index in $\pi_1(X)$. This implies that the index of $\Gamma$ in $\pi_1(X)$ is also finite, 
which means exactly that $\pi_1^{orb}(Y)$ is finite. Hence the lemma. 
\end{proof}

\noindent{\bf Proof of Theorem 1.6}

\vspace{3mm}

{\bf Part (1)}. The proof for this part is based on the rigidity of 
{\it injective Seifert fibered space construction}, which we shall briefly review first, 
see Lee-Raymond \cite{LR0} for more details. 
Suppose we are given a group $\pi$ together with
a short exact sequence $1\rightarrow \Gamma\rightarrow \pi\rightarrow Q\rightarrow 1$, where 
$\Gamma=\Z^k$. Let $W$ be a simply connected smooth manifold and consider the trivial principal
$\R^k$-bundle $\R^k\times W$ over $W$. Let $\psi$ be a smooth, free and properly discontinuous action of $\pi$ on $\R^k\times W$ via bundle morphisms, such that the restriction $\psi|_\Gamma$ 
is given by translations via an embedding $\epsilon: \Gamma=\Z^k\rightarrow \R^k$ as a uniform lattice.  Such an action $\psi$ induces a smooth action of $Q$ on $W$, which is denoted by $\rho$. The quotient space $E\equiv \R^k\times W/\psi(\pi)$ is a Seifert fibered space over the orbifold 
$W/\rho(Q)$, with regular fiber $T^k=\R^k/\epsilon(\Gamma)$ which is a $k$-dimensional torus. 
Conversely, a Seifert fibered space with a regular fiber $T^k$ must arise from such a construction
if the inclusion of a regular fiber induces an injective map on $\pi_1$ (such Seifert fibered spaces
are called injective). In this case the short exact sequence $1\rightarrow \Gamma\rightarrow \pi\rightarrow Q\rightarrow 1$ is part of the homotopy exact sequence associated to the corresponding fibration, with $\pi$ being the $\pi_1$ of the Seifert fibered space, $\Gamma=\Z^k$ being the $\pi_1$ of a 
regular fiber, and $Q$ being the $\pi_1^{orb}$ of the base orbifold. 

Given two such actions $\psi_1,\psi_2$ of $\pi$, with induced embeddings $\epsilon_1, \epsilon_2:
\Gamma\rightarrow \R^k$ and induced actions $\rho_1,\rho_2$ of $Q$ on $W$, the aforementioned 
rigidity theorem asserts that if $\rho_1,\rho_2$ are conjugate by a diffeomorphism $h:W\rightarrow W$,
then $\psi_1,\psi_2$ are conjugate by $(\lambda,g,h)$, where $\lambda\in C^\infty (W,\R^k)$,
$g\in GL(k,\R)$, and 
$$
(\lambda, g,h)\cdot (v,w)=(g(v)+\lambda(h(w)),h(w)), \;\;(v,w)\in \R^k\times W.
$$
Note that in particular, the corresponding Seifert fibered spaces 
$E_1=\R^k\times W/\psi_1(\pi)$ and $E_2=\R^k\times W/\psi_2(\pi)$ are diffeomorphic via 
a fiber-preserving diffeomorphism induced by $(\lambda,g,h)$. See \cite{LR0}, p. 381.

Now let $E_1,E_2$ be two injective Seifert fibered spaces and let $\alpha: \pi_1(E_1)
\rightarrow \pi_1(E_2)$ be an isomorphism. Furthermore, we assume that the universal 
covers of $E_1,E_2$ are diffeomorphic, say given by $\R^k\times W$, and that the isomorphism
$\alpha: \pi_1(E_1)\rightarrow \pi_1(E_2)$ respects the homotopy exact sequences associated 
to the corresponding fibrations on $E_1$ and $E_2$. Note that the latter is always true when 
there is a certain uniqueness of the short exact sequence 
$1\rightarrow \Gamma\rightarrow \pi\rightarrow Q\rightarrow 1$, e.g.,  when $\Gamma=z(\pi)$.
With this understood, we denote the group $\pi_1(E_2)$ by $\pi$ and identify $\pi_1(E)$ with 
$\pi$ via $\alpha$. Then $E_1$, $E_2$ may be regarded as arising from the injective Seifert 
fibered space construction for some actions $\psi_1,\psi_2$ of $\pi$ on $\R^k\times W$. 
Let $\rho_1,\rho_2$ be the induced actions of $Q$ on $W$. Then the rigidity theorem mentioned above implies that there is a fiber-preserving diffeomorphism $\phi: E_1\rightarrow E_2$ such
that $\phi_\ast=\alpha: \pi_1(E_1)\rightarrow \pi_1(E_2)$, if $\rho_1,\rho_2$ are conjugate by a diffeomorphism $h:W\rightarrow W$. (Roughly speaking, the above rigidity theorem allows us to 
show that if the diffeomorphism classification of the base orbifolds are determined by
the fundamental groups, then so are the fiber-preserving diffeomorphism classification of the 
corresponding Seifert fibered spaces.)

With the preceding understood, we shall now give a proof for part (1).  Consider first the case
where $\text{rank }z(\pi_1(X))>1$. A smooth classification of such fixed-point free, smooth 
$\s^1$-four-manifolds is given in Theorem 4.3, which shows that it suffices to consider the case 
where $\text{rank }z(\pi_1(X))=2$ and $\pi_2(X)=0$. Moreover, it also shows that in this case, 
$X$, $X^\prime$ arise from the above injective Seifert fibered space construction with 
$k=2$ and $W=\R^2$. (Note that the uniqueness of the short exact sequence follows from 
the fact that $\Gamma=z(\pi)$, cf. Lemma 2.2 (a)). With this understood, the existence of 
$\phi:X\rightarrow X^\prime$ with $\phi_\ast=\alpha$ follows from the fact that for orientable 
$2$-orbifolds with infinite fundamental group,  any isomorphism 
of $\pi_1^{orb}$ may be realized by a diffeomorphism of the $2$-orbifolds (e.g. see \cite{Mac}).

It remains to consider the case where $\text{rank }z(\pi_1(X))=1$. In this case $X$ is an injective
Seifert fibered space over a $3$-orbifold $Y$ with regular fiber $\s^1$, where $Y$ is an irreducible
$3$-orbifold with infinite fundamental group. As $Y$ is good, the Geometrization Theorem implies
that $Y=\tilde{Y}/G$ for some aspherical $3$-manifold $\tilde{Y}$ (cf. \cite{MM, BLP}). (Note that
$G$ may be trivial here.) Furthermore, by the Geometrization Theorem, $\tilde{Y}$ admits a
geometric decomposition (see e.g. Kleiner-Lott \cite{KL}). In particular,  $\tilde{Y}$ is either Haken, 
or Seifert fibered, or hyperbolic, and the universal cover of $\tilde{Y}$ is diffeomorphic to $\R^3$. 
With this understood, we see that $X$ arises from the injective Seifert fibered space construction 
with $k=1$ and $W=\R^3$. (Note that the condition $\Gamma=z(\pi)$ is satisfied, cf.  Lemma 2.3, which gives the required uniqueness for the short exact sequence 
$1\rightarrow \Gamma\rightarrow \pi\rightarrow Q\rightarrow 1$.) It remains to show that for 
irreducible $3$-orbifolds with infinite fundamental group, any isomorphism of $\pi_1^{orb}$ may be realized by a diffeomorphism of the $3$-orbifolds. This was verified by McCullough and Miller 
(see the proof of Corollary 5.3 in \cite{MM}) when $\tilde{Y}$ is either Haken or Seifert fibered. 
For the remaining case, the $3$-orbifolds are hyperbolic, and in this case, Mostow Rigidity implies that any isomorphism of $\pi_1^{orb}$ may be realized by an isometry of the $3$-orbifolds.
This finishes off the proof for part (1). 

\vspace{3mm}

{\bf Part (2)}. Let $\pi: X\rightarrow Y$ be the orbit map of the $\s^1$-action on $X$.
By Lemma 4.2, this is the case precisely when $Y$ has finite fundamental group. By the Geometrization Theorem, $Y$ is a spherical $3$-orbifold, i.e., there is a finite subgroup $G$ of $SO(4)$ such that $Y=\s^3/G$. Note that the Euler class of  $\pi: X\rightarrow Y$ is torsion, so that there is a $3$-manifold $\hat{Y}$ and a periodic diffeomorphism 
$f$ such that $Y=\hat{Y}/\langle f\rangle$ and $X$ is the mapping torus of $f$. Moreover, 
by the Geometrization Theorem, $\hat{Y}$ is an elliptic $3$-manifold. Similar conclusions
hold for $X^\prime$, i.e., $X^\prime$ is the mapping torus of a periodic diffeomorphism 
$f^\prime$ of an elliptic $3$-manifold $\hat{Y}^\prime$. 

Note that the mapping torus description of $X$, $X^\prime$ implies that $\pi_1(X)$, 
$\pi_1(X^\prime)$ are given by HNN extensions $\pi_1(\hat{Y})\ast_{\pi_1(\hat{Y})} f_\ast$
and $\pi_1(\hat{Y}^\prime)\ast_{\pi_1(\hat{Y}^\prime)} f^\prime_\ast$ respectively. This gives 
rise to short exact sequences 
$$
1\rightarrow\pi_1(\hat{Y})\stackrel{i}{\rightarrow}\pi_1(X) \stackrel{p}{\rightarrow} \Z\rightarrow 1
\mbox{ and }
1\rightarrow\pi_1(\hat{Y}^\prime)\stackrel{i^\prime}{\rightarrow}\pi_1(X^\prime) \stackrel{p^\prime}{\rightarrow} \Z\rightarrow 1.
$$
With this understood, given any isomorphism $\alpha: \pi_1(X)\rightarrow \pi_1(X^\prime)$,
we observe that the homomorphism $p^\prime\circ \alpha\circ i: \pi_1(\hat{Y})\rightarrow \Z$ 
is trivial because $\pi_1(\hat{Y})$ is finite. This implies that $\alpha\circ i: \pi_1(\hat{Y})\rightarrow 
\pi_1(X^\prime)$ lies in the image of $i^\prime: \pi_1(\hat{Y}^\prime)\rightarrow \pi_1(X^\prime)$.
It follows easily from this consideration that $\alpha: \pi_1(X)\rightarrow \pi_1(X^\prime)$ induces 
an isomorphism $\hat{\alpha}: \pi_1(\hat{Y})\rightarrow \pi_1(\hat{Y}^\prime)$ such that 
$f^\prime_\ast= \hat{\alpha}\circ f_\ast\circ \hat{\alpha}^{-1}$ as an element of 
$\text{Out }(\pi_1(\hat{Y}^\prime))$. Suppose $\hat{\alpha}$ 
can be realized by a diffeomorphism $h: \hat{Y}\rightarrow \hat{Y}^\prime$, e.g. 
when $\hat{Y}, \hat{Y}^\prime$ are not lens spaces. Identifying $\hat{Y}$ with
$\hat{Y}^\prime$ via $h$, $X$ may be regarded as the mapping torus of the periodic 
diffeomorphism $g=h\circ f\circ h^{-1}: \hat{Y}^\prime\rightarrow \hat{Y}^\prime$. 
Now observe that $g_\ast=f^\prime_\ast$ as an element of $\text{Out }(\pi_1(\hat{Y}^\prime))$, 
which implies that $g$ and $f^\prime$ are homotopic, hence isotopic 
(cf. \cite{A, Rub, B, HR, BR, BO}). 
The existence of $\phi:X\rightarrow X^\prime$ with $\phi_\ast=\alpha$ follows easily from these considerations. This finishes the proof of part (2). 

We end this section with the smooth classification theorem alluded to earlier. The proof of the theorem employed a key lemma, Lemma 5.2, whose proof will be given in the next section. 

\begin{theorem}
Suppose that $X$ is a fixed-point free, smooth $\s^1$-four-manifold with $\text{rank }z(\pi_1(X))>1$. 
Then $X$ belongs to one of the following cases:
\begin{itemize}
\item [{(1)}] If $\text{rank }z(\pi_1(X))>2$, then $X$ is diffeomorphic to the $4$-torus $T^4$.
\item [{(2)}] If $\text{rank }z(\pi_1(X))=2$ and $\pi_2(X)\neq 0$, then $X$ is diffeomorphic to 
$T^2\times \s^2$.
\item [{(3)}] If $\text{rank }z(\pi_1(X))=2$ and $\pi_2(X)=0$, then $X$ is diffeomorphic to 
$\s^1\times N^3/G$, where $N^3$ is an irreducible Seifert $3$-manifold with infinite fundamental group, and $G$ is a finite cyclic group acting on $\s^1\times N^3$ preserving the product structure and orientation on each factor, and the Seifert fibration on $N^3$.
\end{itemize}
\end{theorem}

\begin{proof}
Let $\pi:X\rightarrow Y$ be the orbit map of the $\s^1$-action. Note that 
$\pi_\ast:\pi_1(X)\rightarrow \pi_1^{orb}(Y)$ is surjective, so that $\pi_\ast (z(\pi_1(X))$
is contained in $z(\pi^{orb}(Y))$. It follows easily from $\text{rank }z(\pi_1(X))>1$
that $z(\pi^{orb}(Y))$ is infinite. By Lemma 5.2, $Y$ is Seifert fibered, and furthermore, 
by Lemma 2.6, $\pi: X\rightarrow Y$ extends to a principal $T^2$-bundle over a $2$-orbifold 
$B$, which will be denoted by $\Pi:X\rightarrow B$. We remark that $B$ is an orientable, 
closed $2$-orbifold.

We begin by describing a decomposition of the principal $T^2$-bundle into a pair of principal 
$\s^1$-bundles over $B$.  More concretely, given any basis $(e_1,e_2)$ of $\pi_1(T^2)$, we 
let $\theta_i:T^2\rightarrow \s^1$, $i=1,2$, be the projections to the first and the second
factor of the decomposition $T^2=\s^1\times \s^1$ that is determined by the basis $(e_1,e_2)$.
This gives rise to a pair of principal $\s^1$-bundles over $B$, denoted by $V_1$, $V_2$, which
are induced by $\theta_1$ and $\theta_2$ respectively. Note that one can recover the principal
$T^2$-bundle $\Pi: X\rightarrow B$ as the pull-back bundle of $V_1\times V_2\rightarrow B\times B$ 
via the diagonal map $B\rightarrow B\times B$. Moreover, with a change of basis,
one can always arrange $V_1$ to have vanishing Euler number. Indeed, under the change 
of basis
$$
e_1 =ae_1^\prime+ce_2^\prime,\;\;\; e_2=be_1^\prime+de_2^\prime, 
$$
where $ad-bc=1$, the corresponding principal $\s^1$-bundles $V_1^\prime, V_2^\prime$ 
associated to the basis $(e_1^\prime,e_2^\prime)$ have Euler numbers
$$
e(V_1^\prime)=a\cdot e(V_1)+b\cdot e(V_2), \;\; e(V_2^\prime)=c\cdot e(V_1)+d\cdot e(V_2). 
$$
If both of $e(V_1)$ and $e(V_2)$ are nonzero, one can choose a unique pair of integers
(up to a sign), $(a,b)$, such that $e(V_1^\prime)=0$. Note that up to a sign, $e(V_2^\prime)$ 
is independent of the choices of $c$ and $d$. This said, we shall assume in what follows that
$e(V_1)=0$.

With these preparations, we now consider case (1) where $\text{rank }z(\pi_1(X))>2$. 
It is clear that $z(\pi_1^{orb}(B))$ is nontrivial and infinite. By Lemma 2.2(a), $B$ must be a 
nonsingular torus. As $e(V_1)=0$ and $B$ is nonsingular, $V_1$ is trivial, which implies that 
$X=\s^1\times V_2$. Finally, the assumption that $\text{rank }z(\pi_1(X))>2$ implies that $V_2$ 
must also be trivial. Hence $X$ is diffeomorphic to the $4$-torus $T^4$.

Consider case (2) where $\text{rank }z(\pi_1(X))=2$ and $\pi_2(X)\neq 0$.  Note that 
$X$ is a principal $\s^1$-bundle over $V_2$ which is the pull-back of the principal $\s^1$-bundle
$V_1\rightarrow B$ via the map $V_2\rightarrow B$. The homotopy exact sequence associated
to the fibration $X\rightarrow V_2$ (cf. Haefliger \cite{Hae2}) implies that $z(\pi_1^{orb}(V_2))$ 
is infinite and 
$\pi_2^{orb}(V_2)\neq 0$. By Lemma 5.2, $V_2$ is the mapping torus of a periodic diffeomorphism 
of a $2$-orbifold $\Sigma$ where $\pi_1^{orb}(\Sigma)$ is finite. Now observe that $e(V_1)=0$ 
implies that $\Sigma$ must be either $\s^2$ or a football. It follows 
easily that $X$ is diffeomorphic to $T^2\times \s^2$, which finishes the proof for case (2).

For case (3) where $\text{rank }z(\pi_1(X))=2$ and $\pi_2(X)=0$, we first observe that 
$\pi_1^{orb}(B)$
is infinite and therefore $B$ is good.  Let $B=\tilde{B}/\Gamma$, where $\tilde{B}$ is a closed 
orientable surface and $\Gamma$ is a finite group acting on $\tilde{B}$. We let
$\tilde{X}$, $\tilde{V}_1$, $\tilde{V}_2$ be the pull-backs of $X\rightarrow B$, $V_1\rightarrow B$, $V_2\rightarrow B$ to $\tilde{B}$ via $\tilde{B}\rightarrow B=\tilde{B}/\Gamma$. Then $\Gamma$ 
acts freely on $\tilde{X}$, giving $X=\tilde{X}/\Gamma$, and $\tilde{V}_1=\s^1\times \tilde{B}$. 
Let $\Gamma_1$ be the subgroup of $\Gamma$ which acts trivially on the $\s^1$-factor in 
$\tilde{V}_1=\s^1\times \tilde{B}$. Then $\Gamma_1$ acts freely on $\tilde{V}_2$. Denote by $N^3$ the quotient $\tilde{V}_2/\Gamma_1$, which is clearly an irreducible Seifert $3$-manifold with 
infinite fundamental group. With this understood, note that $\tilde{X}/\Gamma_1=\s^1\times N^3$, 
so that if we set $G=\Gamma/\Gamma_1$, then  $X=\s^1\times N^3/G$ where the action of $G$ 
preserves the product structure and the orientation of each factor, as well as the Seifert fibration 
on $N^3$, as claimed. This finishes the proof of Theorem 4.3.

\end{proof}

\begin{remark} 
Theorem 2.1 in \cite{C2} asserts that if a $4$-manifold with $b_2^{+}\geq 1$ has nontrivial
Seiberg-Witten invariant, then the homotopy class of the principal orbits of any smooth,
fixed-point free $\s^1$-action on the manifold must be of infinite order; in particular, the center
of the fundamental group must be infinite. As a corollary of Theorem 4.3(2), the converse of the above statement is not true. More concretely, consider a ruled surface $X$
which is a nontrivial $\s^2$-bundle over $T^2$. Note that $X$ satisfies $b_2^{+}\geq 1$, has
nontrivial Seiberg-Witten invariant, and $z(\pi_1(X))$ is infinite. However, by Theorem 4.3(2) $X$ does not admit any smooth, fixed-point free $\s^1$-action. It is also interesting to note that
a double cover of $X$, which is diffeomorphic to $\s^2\times T^2$, admits a smooth,
fixed-point free $\s^1$-action. We remark that for a closed aspherical manifold, such a correlation between the existence of circle actions and the nontriviality of the center of the fundamental group is part of a conjectured rigidity of aspherical manifolds going back to work of Borel. See \cite{CWY} for some recent progress and more detailed discussions.

\end{remark}

\section{Injectivity of $\s^1$-actions when $\pi_1$ has infinite center}

The main purpose of this section is to show that a smooth fixed-point free $\s^1$-four-manifold whose 
fundamental group has infinite center is injective, hence admits a fiber-sum decomposition. 
A key role is played by Lemma 5.2, whose proof requires the use of the Geometrization Theorem 
in various forms. 

We begin with the following observation. 

\begin{lemma}
Let $Y$ be a $3$-orbifold with a singular set consisting of a union of circles. Then
there is a good $3$-orbifold $Y_0$ such that $Y$, $Y_0$ have the same 
underlying space, and $\pi_1^{orb}(Y_0)=\pi_1^{orb}(Y)$.
\end{lemma}

\begin{proof}
Denote by $|Y|$ the underlying $3$-manifold of $Y$ and by $\Sigma Y$ the singular set of $Y$,
consisting of components $\gamma_1,\cdots, \gamma_n$. Then $\pi_1^{orb}(Y)$ admits the 
following presentation
$$
\pi_1^{orb}(Y)=\pi_1(|Y|\setminus \Sigma Y)/N.
$$
Here $N$ is the normal subgroup generated by the elements $\mu_{\gamma_i}^{m_i}$,
$i=1,2,\cdots,n$, where $\mu_{\gamma_i}$ is the meridian around $\gamma_i$ and $m_i$ is the multiplicity of $\gamma_i$ (cf. \cite{BMP}, Proposition 2.7). 

With this understood, for any bad $2$-suborbifold $C$ in $Y$, one has the following two
possibilities: 

\begin{itemize}
\item [{(i)}] there is exactly one $\gamma_i$ such that $C\cap \gamma_i\neq \emptyset$;
\item [{(ii)}] there are $\gamma_i,\gamma_j$, $i\neq j$, $m_i\neq m_j$, such that 
$C\cap \gamma_i\neq \emptyset$, $C\cap \gamma_j\neq \emptyset$.
\end{itemize}

In case (i), the existence of such a $C$ implies that $\mu_{\gamma_i}=1$ in $\pi_1(|Y|\setminus
\Sigma Y)$, hence $\pi_1^{orb}(Y)$ is unchanged after removing $\gamma_i$ from $\Sigma Y$.
In the resulting $3$-orbifold, $C$ is no longer a bad $2$-suborbifold. 

In case (ii), let $m=\text{gcd }(m_i,m_j)$. We change $Y$ to a new $3$-orbifold by replacing 
the multiplicities of $\gamma_i$, $\gamma_j$ with $m$. (In case of $m=1$, this simply means 
that $\gamma_i,\gamma_j$ are both removed from $\Sigma Y$.) Note that the existence of $C$
implies that the normal subgroup generated by $\mu_{\gamma_i}^{m_i}$ and 
$\mu_{\gamma_j}^{m_j}$ is the same as that generated by $\mu_{\gamma_i}^{m}$ and 
$\mu_{\gamma_j}^{m}$. It follows that $\pi_1^{orb}(Y)$ remains unchanged in this process.
Since there are only finitely many singular circles and during the process either the 
number of singular circles is decreased or the multiplicity of a singular circle is
decreased, this process must terminate in finitely many steps. At the end, we obtain a good $3$-orbifold $Y_0$ such that $|Y_0|=|Y|$ and $\pi_1^{orb}(Y_0)=\pi_1^{orb}(Y)$. Hence the lemma.

\end{proof}

A more conceptual view which was suggested by the referee goes as follows: introducing a notion of
complexity for $3$-orbifolds, say by the sum of the multiplicities of the singular circles, then the orbifold
$Y_0$ in Lemma 5.1 is characterized as the one with the minimal complexity among the $3$-orbifolds 
which have the same underlying space and the same fundamental group of the orbifold $Y$. 

In the following lemma, for the definition of $\pi_2^{orb}(Y)$ we refer to \cite{Hae1, Hae2, C0}.

\begin{lemma}
Let $Y$ be an orientable $3$-orbifold, not necessarily good, with a singular set
consisting of a union of circles. If $z(\pi_1^{orb}(Y))$ is infinite, then $Y$ is Seifert fibered.
Moreover, if $\pi_2^{orb}(Y)\neq 0$, then $Y$ is the mapping torus of a periodic diffeomorphism 
of a $2$-orbifold with finite fundamental group.
\end{lemma}

\begin{proof}
Let $Y_0$ be the good $3$-orbifold associated to $Y$ from Lemma 5.1, which is clearly
orientable. Then there is an orientable $3$-manifold $Y^\prime$ equipped with a finite group action 
of $G$, such that $Y_0=Y^\prime/G$ (cf. \cite{BLP, MM}). Since $\pi_1^{orb}(Y_0)=\pi_1^{orb}(Y)$, 
$z(\pi_1^{orb}(Y_0))$ is also infinite, and consequently, $z(\pi_1(Y^\prime))$,
which contains $\pi_1(Y^\prime)\cap z(\pi_1^{orb}(Y_0))$, is infinite. 
As an abelian subgroup of a $3$-manifold group, $z(\pi_1(Y^\prime))$ must contain 
an infinite cyclic subgroup $H$ (cf. \cite{Hem}, Theorem 9.14), which is clearly normal 
in $\pi_1(Y^\prime)$. 

Consider first the case where $\pi_2(Y^\prime)=0$. By work of Gabai (cf. \cite{Gabai}, and independently, Casson-Jungreis \cite{CJ}), $Y^\prime$ is Seifert fibered, with $H$ 
being generated by a regular fiber of the Seifert fibration. Since $H\subset z(\pi_1^{orb}(Y_0))$, 
it must be invariant under the action of $G$. By a theorem of Meeks and Scott (cf. \cite{MS}, 
Theorem 2.2), $G$ preserves the Seifert fibration on $Y^\prime$, which implies
that $Y_0$ is Seifert fibered. Since we assume $\pi_2(Y^\prime)=0$, $Y_0$ does not contain
any essential spherical $2$-suborbifold. From the proof of Lemma 5.1, we see that $Y$ contains
no bad $2$-suborbifold, and in this case, $Y=Y_0$. This proves that $Y$ is Seifert fibered.
Note that in this case, 
$$
\pi_2^{orb}(Y)=\pi_2^{orb}(Y_0)=\pi_2(Y^\prime)=0.
$$

Suppose $\pi_2(Y^\prime)\neq 0$. Since $z(\pi_1(Y^\prime))$ is nontrivial, $Y^\prime$ must be
prime (here we use Lemma 2.1 and the resolution of the Poincar\'{e} conjecture \cite{P}), and consequently, $Y^\prime=\s^1\times \s^2$. Note that $G$ must act on $Y^\prime=\s^1\times\s^2$ homologically trivially because the fundamental group of $Y_0=Y^\prime/G$ is infinite. 
By Lemma 2.5, $Y_0=Y^\prime/G$ is the mapping torus of a periodic diffeomorphism of some
spherical $2$-orbifold; in particular, $Y_0$ is Seifert fibered. If $Y$ is good, then 
$Y=Y_0$, and the lemma follows in this case. Note that in this case,
$$
\pi_2^{orb}(Y)=\pi_2^{orb}(Y_0)=\pi_2(Y^\prime)\neq 0.
$$

It remains to consider the case where $Y$ is not good. Recall that in the proof of 
Lemma 5.1, $Y_0$ is obtained from $Y$ by performing a sequence of operations in each 
of which either a singular circle is removed or its multiplicity is decreased. 
Since $Y_0$  is the mapping torus of a periodic diffeomorphism $f$ of some spherical 
$2$-orbifold $\Sigma$, it follows easily that $\Sigma$ is either $\s^2$ or a football. 
Moreover, if $\Sigma$ is a football, $f$ must be isotopic to the identity map, and therefore
$Y_0$ is diffeomorphic to $\s^1\times \Sigma$.  It follows readily that $Y$ is the product 
of $\s^1$ with a bad $2$-orbifold $B$. Note that in this case,
$$
\pi_2^{orb}(Y)=\pi_2^{orb}(B)\neq 0
$$
since a bad $2$-orbifold has nontrivial $\pi_2^{orb}$. 

Suppose $\Sigma=\s^2$, and therefore $Y_0=\s^1\times\s^2$. Note that $Y$ can have 
at most two singular circles. Assume first that $Y$ has only one singular circle, 
which is denoted by $\gamma$.
It suffices to show that $(|Y|,\gamma)$ and $(\s^1\times \s^2, \s^1\times \{pt\})$ are
diffeomorphic. To see this, let $W=Y\setminus Nd(\gamma)$ and let $\mu$
denote a meridian of $\gamma$. Then $\pi_1^{orb}(Y)=\pi_1(W)/\langle \mu^m\rangle$
where $m$ denotes the multiplicity of $\gamma$. Since $\mu$ bounds a disc in $W$, and
$\pi_1^{orb}(Y)=\pi_1(Y_0)=\Z$, it follows that $\pi_1(W)=\Z$. Cutting $W$ open along
the disc bounded by $\mu$, we obtain a $3$-manifold $W_0$ with $\partial W_0=\s^2$
and $\pi_1(W_0)$ trivial. By the Geometrization Theorem, $W_0$ is a $3$-ball, which implies
easily that $(|Y|,\gamma)$ is diffeomorphic to $(\s^1\times \s^2, \s^1\times \{pt\})$. 
This shows that $Y$ is the product of $\s^1$ with a teardrop.  Note that $\pi_2^{orb}(Y)\neq 0$
as we argued before.

Finally, suppose $Y$ has two components, denoted by $\gamma_1,\gamma_2$, which
have multiplicities $m_1, m_2$ respectively. From the construction of $Y_0$ in Lemma 5.1,
it follows easily that $m_1,m_2$ are relatively prime. With this understood, it suffices to 
show that $(|Y|,\gamma_1,\gamma_2)$ is diffeomorphic to 
$(\s^1\times \s^2, \s^1\times \{pt\},\s^1\times \{pt\})$. First, as we argued in the previous case,  
$(|Y|,\gamma_1)$ is diffeomorphic to $(\s^1\times \s^2, \s^1\times \{pt\})$, so that if we let 
$W=Y\setminus Nd(\gamma_1)$, then $|W|=\s^1\times D^2$.
It remains to show that $(|W|,\gamma_2)$ is diffeomorphic to $(\s^1\times D^2, \s^1\times \{pt\})$. 
To see this, note that the meridians $\mu_1,\mu_2$ of $\gamma_1,\gamma_2$ bound an 
annulus in $W\setminus Nd(\gamma_2)$. Consequently, 
$$
\Z=\pi_1^{orb}(Y)=\pi_1(W\setminus Nd(\gamma_2))/\langle \mu_1^{m_1},\mu_2^{m_2}\rangle=
\pi_1(W\setminus Nd(\gamma_2))/\langle \mu_2\rangle,
$$
which implies the short exact sequence 
$$
1\rightarrow \Z_{m_2}\rightarrow \pi_1^{orb}(W)=\pi_1(W\setminus Nd(\gamma_2))/\langle \mu_2^{m_2}\rangle\rightarrow \Z\rightarrow 1.
$$
Now if we cut $W$ open along a copy of $\{pt\}\times D^2$ in $|W|=\s^1\times D^2$, 
we obtain a $3$-orbifold $W_0$ with $\partial W_0=\s^2/\Z_{m_2}$. Moreover, it follows from the above short exact sequence that $\pi_1^{orb}(W_0)=\Z_{m_2}$. Then the Geometrization 
Theorem implies that $W_0$ is discal, from which it follows that 
$(|Y|,\gamma_1,\gamma_2)$ is diffeomorphic to 
$(\s^1\times \s^2, \s^1\times \{pt\},\s^1\times \{pt\})$, and consequently, $Y$ is the product
of $\s^1$ with a bad $2$-orbifold. Moreover, $\pi_2^{orb}(Y)\neq 0$.
This finishes the proof of the lemma.

\end{proof}

\noindent{\bf Proof of Theorem 1.4}

\vspace{3mm}

Let $\pi: X\rightarrow Y$ be the orbit map of the fixed-point free $\s^1$-action. 
Suppose the $\s^1$-action is not injective. Then the homotopy class of a regular fiber of 
$\pi$ is finite, and since $z(\pi_1(X))$ is infinite, the image of $z(\pi_1(X))$ under 
$\pi_\ast: \pi_1(X)\rightarrow \pi_1^{orb}(Y)$, clearly contained in 
$z(\pi_1^{orb}(Y))$, must also be infinite. By Lemma 5.2, either $Y$ is irreducible, or
$Y$ is the mapping torus of a periodic diffeomorphism of a $2$-orbifold with finite 
fundamental group. Since we assume that the homotopy class of a regular fiber of $\pi$ 
is finite, $Y$ can not be irreducible. Then it follows easily that $X$ is the mapping torus
of a periodic diffeomorphism of some elliptic $3$-manifold. 

To see that $X$ admits a fiber-sum decomposition, it suffices to consider the case where
the $\s^1$-action is injective. We note first that the fact that the homotopy class 
of a regular fiber of $\pi$ has infinite order implies that the orbit space $Y$ of the $\s^1$-action
does not contain any bad $2$-suborbifolds. In other words, $Y$ must be good. By Lemma 2.4, 
$Y$ admits a reduced spherical decomposition. More precisely, there is a system of finitely 
many spherical $2$-suborbifolds $\Sigma_j\subset Y$, such that after capping off the 
boundary of each component of $Y\setminus \cup_j \Sigma_j$, one obtains a collection 
of $3$-orbifolds $Y_i$,  where each $Y_i$ is irreducible. Furthermore, each $\Sigma_j$ 
must be either an ordinary $2$-sphere or a football, and the pre-image 
$N_j\equiv \pi^{-1}(\Sigma_j)$ must be diffeomorphic to $\s^1\times \s^2$, because the 
homotopy class of a regular fiber of $\pi$ has infinite order. Finally, observe that the restriction
of $\pi$ on each $N_j$ may be uniquely extended to a Seifert-type $\s^1$-fibration on 
$\s^1\times B^3$, so that correspondingly, we obtain the irreducible $\s^1$-four-manifolds 
$X_i$ and the orbit maps $\pi_i: X_i\rightarrow Y_i$. It follows easily that $X$ is 
fiber-sum-decomposed into $X_i$ along $N_j$. We remark that the requirement that the spherical decomposition of $Y$ be reduced ensures that Definition 1.3(iv) is satisfied. This finishes 
off the proof of Theorem 1.4.

\vspace{3mm}

\noindent{\bf Proof of Corollary 1.7}

\vspace{3mm}

By Theorem 1.4, it suffices to consider the case where the $\s^1$-action is injective.
Let $\pi:X\rightarrow Y$ be the corresponding orbit map. We observe that 
$Y$ does not contain any bad $2$-suborbifolds, hence there exist a $3$-manifold $\tilde{Y}$ 
and a finite group $G$ such that $Y=\tilde{Y}/G$ (cf. \cite{BLP,MM}). On the other hand, by the 
homotopy exact sequence associated to $\pi:X\rightarrow Y$ (cf. Haefliger \cite{Hae2}), it follows easily that $\pi_\ast: \pi_2(X)\rightarrow \pi_2^{orb}(Y)$ is an isomorphism.  
Let $\tilde{\pi}: \tilde{X}\rightarrow \tilde{Y}$
be the pull-back fibration via the projection $\tilde{Y}\rightarrow Y$. Then $\tilde{X}$ is a finite
regular cover of $X$. It suffices to show that there exist no
embedded $2$-spheres with odd self-intersection in $\tilde{X}$. 

Suppose to the contrary, there is an embedded $2$-sphere $C$ in $\tilde{X}$ with 
$C^2\equiv 1\pmod{2}$. Consider the projection of $C$ into $\tilde{Y}$ under $\tilde{\pi}$. 
Clearly $[C]\in \pi_2(\tilde{X})$ is nonzero. On the other hand, 
$\pi_\ast: \pi_2(X)\rightarrow \pi_2^{orb}(Y)$ is an isomorphism, so that 
$\tilde{\pi}_\ast: \pi_2(\tilde{X})\rightarrow \pi_2(\tilde{Y})$ is also an isomorphism. 
Consequently, $\tilde{\pi}|_{C}: \s^2\rightarrow \tilde{Y}$ is homotopically nontrivial.
By the Sphere Theorem (cf. \cite{Hem}, Theorem 4.11), there is an embedded $2$-sphere 
$\Sigma$ in a neighborhood of $\tilde{\pi}(C)$, whose class is clearly homologous 
to $\tilde{\pi}_\ast [C]$.  Observe that the Euler class of $\tilde{\pi}:\tilde{X}\rightarrow \tilde{Y}$ evaluates to $0$ on $\Sigma$. 
This is because the pull-back of the Euler class of $\tilde{\pi}$ to $\tilde{X}$ is zero so that the Euler 
class of $\tilde{\pi}$ evaluates trivially on the class of $\tilde{\pi}(C)$. This implies that the restriction 
of $\tilde{\pi}$ to $\Sigma$ is trivial, and in particular, $\Sigma$ has a section $\Sigma^\prime$ in 
$\tilde{X}$. Consequently, we obtain an equation of homology classes
$$
C=\Sigma^\prime+ \sum_i T_i
$$
where $T_i=\tilde{\pi}^{-1}(\gamma_i)$ for some loops $\gamma_i\subset \tilde{Y}$ (cf. \cite{Bald},
Theorem 9). Since $\Sigma^\prime$, $T_i$ all have self-intersection $0$, this implies 
$C^2\equiv 0 \pmod{2}$ which is a contradiction. This finishes the proof of Corollary 1.7.

\section{Theorems 1.1 and 1.2}

This section is devoted to the proofs of Theorems 1.1 and 1.2. We remark that while Theorem 1.1 follows 
readily from Theorems 1.5 and 1.6, the proof of Theorem 1.2 requires some additional care in the case
when all the irreducible $\s^1$-four-manifolds in the fiber-sum decomposition are a mapping torus of
a periodic diffeomorphism of a lens space. Furthermore, the case when the fundamental group 
of the $4$-manifold is isomorphic
to the fundamental group of a Klein bottle needs to be dealt with separately. In all these considerations,
the following lemma describing certain isotopies of periodic diffeomorphisms of $\s^3$ or a lens space
played a key role. 

Let $Y=\s^3/G$ where $G$ is a cyclic subgroup of $SO(4)$ of order $n$ given by 
$$
\lambda \cdot (z_1,z_2)=(\lambda^{p}z_1,\lambda^{q}z_2), 
$$
where $\lambda=\exp(2\pi i/n)$ is a $n$-th root of unity and $\gcd(n,p,q)=1$. 
Set $u=\gcd(n,p)$, $v=\gcd(n,q)$. Then $\gcd(u,v)=1$ so that $uv$ 
is a divisor of $n=|\pi_1^{orb}(Y)|$, and $Y$ has at most two singular circles 
of multiplicities $u,v$, given by $z_2=0$ and $z_1=0$ respectively. 

Suppose $H$ is a subgroup of $G$ of order $\hat{n}$ generated by $\lambda^{n/\hat{n}}$, which
acts freely on $\s^3$. Note that this condition is equivalent to $\gcd(\hat{n},p)=1$ and 
$\gcd(\hat{n},q)=1$; in particular, $\hat{n},u,v$ are pairwise co-prime so that $\hat{n}\leq n/uv$. 
We set $\hat{Y}=\s^3/H$, which is either $\s^3$ or a lens space. With this understood, let 
$f:\hat{Y}\rightarrow \hat{Y}$ be a periodic diffeomorphism such that $Y=\hat{Y}/\langle f\rangle$.

\begin{lemma}
For any singular circle $\gamma$ of $Y$, say the one defined by $z_2=0$ which has multiplicity 
$u$, we let $\hat{\gamma}$ be the pre-image of $\gamma$ in $\hat{Y}$. Then there exist a periodic diffeomorphism $f^\prime: \hat{Y}\rightarrow \hat{Y}$ and an isotopy $f_t: \hat{Y}\rightarrow \hat{Y}$ between $f$ and $f^\prime$, such that
\begin{itemize}
\item the restriction of $f_t$ on $\hat{\gamma}$ is independent of $t$; in particular, $f=f^\prime$ on
$\hat{\gamma}$;
\item $f^\prime$ is free on $\hat{\gamma}$ so that the image of $\hat{\gamma}$ in $Y^\prime=
\hat{Y}/\langle f^\prime \rangle$ is not a singular circle;
\item when $\hat{Y}=\s^3$, one can arrange $f^\prime$ such that $Y^\prime$ is the lens space 
$L(n/u,1)$. 
\end{itemize}
\end{lemma}

\begin{proof}
We first consider the case where $\hat{n}>1$. Set $p^\prime=p/u$, and let $u^\prime$ be the unique
integer satisfying $uu^\prime\equiv 1 \pmod{\hat{n}}$ and $0<u^\prime<\hat{n}$, and consider the following
action of a cyclic subgroup $G^\prime\subset SO(4)$ of order $n^\prime=n/u$, given by
$$
\delta\cdot (z_1,z_2)=(\delta^{p^\prime} z_1,\delta^{qu^\prime}z_2),
$$
where $\delta=\exp (2\pi i/n^\prime)$ is a $n^\prime$-th root of unity. Note that since 
$\lambda^{n/\hat{n}}\cdot (z_1,z_2)=\delta^{n^\prime u/\hat{n}}\cdot (z_1,z_2)$, 
$H=\langle \lambda^{n/\hat{n}}\rangle=\langle \delta^{n^\prime/\hat{n}}\rangle$ 
is also a subgroup of $G^\prime$. 

There is a $k$ with $\gcd(n,k)=1$ such that $f: \hat{Y}\rightarrow \hat{Y}$ is represented 
by the $H$-equivariant map $F: (z_1,z_2)\mapsto \lambda^k\cdot (z_1,z_2)$. We shall consider the 
$H$-equivariant map $F^\prime: (z_1,z_2)\mapsto \delta^k\cdot (z_1,z_2)$, which has the 
following properties: (i) $F=F^\prime$ on $\{(z_1,0)||z_1|=1\}$, (ii) there is a $H$-equivariant 
isotopy $F_{t}$ between $F$ and $F^\prime$ which is constant in $t$ on $\{(z_1,0)||z_1|=1\}$. 
For instance, $F_t: (z_1,z_2)\mapsto (\delta^{kp^\prime}z_1, \theta_t z_2)$, where 
$\theta_t=\exp (2tkqu^\prime\pi i /n^\prime+2(1-t)kq\pi i/n)$, $0\leq t\leq 1$. Let $f^\prime, f_t$ be the
descendant of $F^\prime, F_t$ to $\hat{Y}$ respectively. Then clearly $f_t$ is an isotopy between
$f$ and $f^\prime$, which is constant on $\hat{\gamma}=\{(z_1,0)||z_1|=1\}/H$, and $f^\prime$ is 
free on $\hat{\gamma}$ so that the image of $\hat{\gamma}$ in 
$Y^\prime=\hat{Y}/\langle f^\prime \rangle$ is not a singular circle. This finishes the proof for the case
where $\hat{n}>1$.

Now suppose $\hat{n}=1$, which means that $H$ is trivial. Then instead, we consider the following 
action of a cyclic subgroup $G^\prime\subset SO(4)$ of order $n^\prime=n/u$, given by
$$
\delta\cdot (z_1,z_2)=(\delta^{p^\prime} z_1,\delta^{p^\prime}z_2).
$$
The rest of the argument is the same, with $H\subset G^\prime$ trivially true. 
Note that in this case, $Y^\prime=\s^3/\langle f^\prime\rangle=\s^3/G^\prime=L(n/u,1)$. This finishes
the proof of Lemma 6.1.

\end{proof} 

As an immediate corollary of Lemma 6.1, we obtain the following classification of 
fixed-point free smooth $\s^1$-four-manifolds whose fundamental group is isomorphic to the fundamental group of a Klein bottle. 

\begin{theorem}
Let $X$ be a fixed-point free smooth $\s^1$-four-manifold such that $\pi_1(X)$ is isomorphic 
to the fundamental group of a Klein bottle. Then $X$ is diffeomorphic to the quotient of 
$T^2\times\s^2$ by the involution $\tau$, where 
$$
\tau: (x,y,z)\mapsto (-x,\bar{y},-z) \mbox{ for } x,y \in \s^1\subset\C \mbox{ and } z\in\s^2\subset \R^3.
$$
\end{theorem}

\begin{proof}

As $\pi_1(X)$ is isomorphic to the fundamental group of a Klein bottle, it has the following
presentation: $\pi_1(X)=\{c,t| tct^{-1}=c^{-1}\}$. Clearly the center $z(\pi_1(X))$ is the infinite
cyclic subgroup generated by $t^2$. By Theorem 1.4, the $\s^1$-action is injective. We let
$\pi:X\rightarrow Y$ be the corresponding orbit map. Let $m>0$ be the multiplicity of the 
homotopy class of a regular fiber of $\pi$ in $z(\pi_1(X))$. Then 
$$\pi_1^{orb}(Y)=\{c,t|tct^{-1}=c^{-1}, t^{2m}=1\}.$$ 

Let $\hat{Y}$ be the regular covering of $Y$ corresponding to the infinite normal cyclic subgroup
generated by $c$. Since $\hat{Y}$ is good and its fundamental group is torsion-free,
$\hat{Y}$ must be a $3$-manifold, and clearly, $\hat{Y}=\s^1\times \s^2$. The corresponding
group of deck transformations on $\hat{Y}$ is cyclic of order $2m$, generated by $t$ which 
sends $c\in \pi_1(\hat{Y})$ to $-c$. By Lemma 2.5, $Y$ is diffeomorphic to either 
$\R\P^3_m \#_{m} \R\P^3_m$, or $\R\P^3_m \#_{m} \widetilde{\R\P^3}_m$, 
or $\widetilde{\R\P^3}_m\#_{m} \widetilde{\R\P^3}_m$. Consequently, $X$ is fiber-sum-decomposed 
into $X_1,X_2$ along $N$, with $\pi_i:X_i\rightarrow Y_i$, $i=1,2$, where $Y_1,Y_2$ is either 
$\R\P^3_m$ or $\widetilde{\R\P^3}_m$, and $\pi: N\rightarrow \Sigma$ where $\Sigma$ intersects
the singular circle of multiplicity $m$ in $Y$. 

There are $\hat{Y}_i$ and periodic diffeomorphisms $f_i$ such that $Y_i=\hat{Y}_i/\langle
f_i\rangle$ and $X_i$ is the mapping torus of $f_i$, where $i=1,2$. 
We apply Lemma 6.1 to $Y_i$, $\hat{Y}_i$ and $f_i$, with $\gamma$ being the singular 
circle of multiplicity $m$. We claim that in either case, i.e., $Y_i=\R\P^3_m$ or 
$\widetilde{\R\P^3}_m$, $\hat{Y}_i$ must be $\s^3$, i.e., $\hat{n}=1$. For the case where
$Y_i=\widetilde{\R\P^3}_m$, it follows from the fact that $\widetilde{\R\P^3}_m$
has two singular circles with multiplicities $2,m$ respectively, so that $\hat{n}\leq n/uv=2m/2m=1$.
For the case where $Y_i=\R\P^3_m$, a similar argument shows that $\hat{n}\leq 2$. 
Continuing using the notations in Lemma 6.1, we have, in this case, $p=m$, $q=1$, 
$n^\prime=2$, and $f_i^\prime$ is given by multiplication by $\delta$. If $\hat{n}=2$ and 
therefore $\hat{Y}_i=\R\P^3$, $f_i^\prime$ is the identity map on $\hat{Y}_i$.
Consequently, as the mapping torus of $f_i^\prime$, $X_i$ is diffeomorphic to
$\s^1\times \hat{Y}_i$, and $\pi_1(X)$ contains a torsion subgroup of $\Z_2$ coming 
from $\pi_1(\hat{Y}_i)$. But this contradicts the fact that $\pi_1(X)$ is isomorphic to the 
$\pi_1$ of a Klein bottle. Hence $\hat{Y}_i=\s^3$ in both cases. We conclude by observing 
that each $X_i$ is the mapping torus of the antipodal map on $\s^3$. We denote by
$\pi_i^\prime: X_i\rightarrow \R\P^3$ the corresponding Seifert-type $\s^1$-fibration. 

Finally, by the property in Lemma 6.1 that the restriction of $f_t$ on $\hat{\gamma}$ is independent 
of $t$, it is easily seen that the Seifert-type $\s^1$-fibrations $\pi_i:X_i\rightarrow Y_i$ and 
$\pi_i^\prime: X_i\rightarrow \R\P^3$ are identical on
the mapping torus of $f=f^\prime: \hat{\gamma}\rightarrow \hat{\gamma}$. It follows easily
that $X$ is also fiber-sum-decomposed into $X_1,X_2$ along $N$, with $\pi_i^\prime: 
X_i\rightarrow \R\P^3$ on each factor $X_i$. Theorem 6.2 follows easily. 

\end{proof}

Theorem 1.1 follows immediately from the following theorem. 

\begin{theorem}
Let $G$ be a finitely presented group such that (i) $\text{rank }z(G)=1$, (ii)
$G$ is single-ended and is not isomorphic to the $\pi_1$ of 
a Klein bottle, (iii) any canonical JSJ decomposition of $G$ contains a vertex 
subgroup which is not isomorphic to an HNN extension of a finite cyclic group.
Let $S_G$ be the set of equivariant diffeomorphism classes of orientable, fixed-point free,
smooth $\s^1$-four-manifolds $X$ such that $\pi_1(X)=G$. Then there exists a constant
$C>0$ depending only on $G$, such that $\# S_G<C$.
\end{theorem}

\begin{proof}
Let $X$ be an orientable, fixed-point free, smooth $\s^1$-four-manifolds $X$ such that 
$\pi_1(X)=G$. Since $G$ is single-ended, it follows easily from Theorem 1.4 that any fixed-point free
$\s^1$-action on $X$ must be injective. Thus any fixed-point free $\s^1$-action on $X$ is associated with 
a canonical fiber-sum decomposition. 
Suppose $X$ is decomposed into factors $X_i$ along $N_j$. For convenience we shall fix an orientation
of $X$, which is the one induced from the fiber-sum decomposition. Then the following data completely determine the oriented equivariant diffeomorphism class of $X$:

\begin{itemize}
\item [{(i)}] The isomorphism class of the underlying graph of $\Lambda$.
\item [{(ii)}] For each pair of $i,j$ such that $N_j\subset X_i$, the fiber-preserving isotopy 
class of embeddings of $N_j$ in $X_i$ for each fixed  oriented, fiber-preserving
diffeomorphism class of $X_i$.
\item [{(iii)}] For each $i$, the oriented, fiber-preserving diffeomorphism class of $X_i$.
\end{itemize}

These data are subject to the following constraints: the cardinalities of $\{X_i\}$ and $\{N_j\}$ 
and the conjugacy classes of subgroups $\pi_1(X_i)$, $\pi_1(N_j)$ in $G$ are determined by 
$G$ (cf. Proposition 3.5). With this understood, our aim is to show that the number of objects in 
each of (i), (ii), and (iii) is bounded by a constant depending only on $G$. 

The number of objects in (i) is clearly bounded by a constant depending only on $G$, since 
the cardinalities of $\{X_i\}$ and $\{N_j\}$ are fixed by $G$. For the objects in (ii) and (iii), where
an index $i$ is being fixed, we shall discuss separately according to the following three cases, 
(a) $\text{rank }z(\pi_1(X_i))>1$, (b) $\pi_1(X_i)$ is single-ended with $\text{rank }z(\pi_1(X_i))=1$,
(c) $\pi_1(X_i)$ is double-ended. 

Note that the number of objects in (ii) is bounded by the number of singular circles of
$Y_i$ plus one, so we need to show that for each $i$, the number of singular circles of
$Y_i$ is bounded by a constant depending only on $G$. With this understood, consider 
case (a) where $X_i$ is a Seifert-type $T^2$-fibration over a $2$-orbifold $B_i$ with infinite 
$\pi_1^{orb}$. As shown in the proof of Theorem 1.6(1), $B_i$ is uniquely determined by 
$\pi_1(X_i)$,  hence by $G$. On the other hand, $Y_i$ is Seifert fibered over $B_i$, 
so that the number of singular circles of $Y_i$ is bounded by the number of singular 
points of $B_i$, which depends only on $G$. In case (b), $Y_i$ is uniquely determined by 
$\pi_1(X_i)$ as shown in the proof of 
Theorem 1.6(1), hence the number of singular circles of $Y_i$ depends only on $G$.
In case (c),  $\pi_1^{orb}(Y_i)$ is finite. The Geometrization Theorem implies that $Y_i$ is 
spherical. Since the singular set of $Y_i$ consists of a union of embedded circles, 
the work of Dunbar in \cite{Dun} shows that $Y_i=\s^3/G_i$ where $G_i$ is a subgroup of
$SO(4)$ which preserves a Hopf fibration. It follows easily that the number of singular 
components of $Y_i$ is universally bounded (say by $4$). This shows that the number of 
objects in (ii) is bounded by a constant depending only on $G$. 

Finally, we examine the boundedness of the number of objects in (iii). In case (a), the 
diffeomorphism class of $X_i$ is uniquely determined by $\pi_1(X_i)$ (cf. Theorem 1.6(1)),
however, the Seifert-type $\s^1$-fibration $\pi_i: X_i\rightarrow Y_i$ has infinitely many choices, one 
for each primitive element of $z(\pi_1(X_i))$. With this understood, note that by assumption
$z(G)$ has rank $1$, so that there is only one possible choice for the regular fiber class of 
$\pi_i$ in $z(\pi_1(X_i))$. This shows that $\pi_i:X_i\rightarrow Y_i$ is uniquely determined by 
$G$ in this case. In case (b), both $X_i$ and $\pi_i$ are uniquely determined by $\pi_1(X_i)$ 
as shown in the proof of Theorem 1.6(1), hence are also determined by $G$.  

Lastly, we consider case (c). By Theorem 1.6(2), $X_i$ is the mapping torus of a periodic 
diffeomorphism $f_i:\hat{Y}_i\rightarrow \hat{Y}_i$ of an elliptic $3$-manifold. 
It follows from the proof of Theorem 1.6(2) that the number of diffeomorphism 
classes of $X_i$ is bounded by a constant depending only on $\pi_1(X_i)$. In order to 
bound the number of fiber-preserving diffeomorphism classes, we shall employ the rigidity
theorem of injective Seifert fibered space construction as in the proof of Theorem 1.6(1),
with $k=1$ and $W=\s^3$. With this understood, it is clear that it suffices to show that the
number of possible short exact sequences 
$1\rightarrow \Gamma\rightarrow \pi\rightarrow Q\rightarrow 1$ involved in the argument 
is bounded by a constant depending only on $G$. Equivalently, we will show that the 
multiplicity of the homotopy class of a regular fiber of $\pi_i$ in $z(\pi_1(X_i))$ is bounded 
by a constant depending only on $G$.

Denote by $h$ the homotopy class of a regular fiber. Since the conjugacy 
classes of the subgroups $\pi_1(X_i)$ in $G$ depend only on $G$, it follows easily that 
it suffices to bound the multiplicity of $h$ in $z(\pi_1(X))$. With this understood, we observe 
that since for each $j$, $z(\pi_1(X))\subset \pi_1(N_j)$, the multiplicity of $h$ in $z(\pi_1(X))$ 
is bounded by the multiplicity of $h$ in $\pi_1(N_j)$ for every $j$, which equals $1$ if 
$\Sigma_j$ is an ordinary $2$-sphere, and equals the multiplicity of the singular circle 
of $Y$ that $\Sigma_j$ intersects otherwise. In particular, if one of the $\Sigma_j$ is an ordinary 
$2$-sphere, or one of the $Y_i$ has infinite fundamental group, we are done for (iii).
(Note that since $G$ is single-ended, there is at least one $N_j$ if case (c) is valid.)

Suppose $\pi_1^{orb}(Y_i)$ is finite for each $i$ and $\Sigma_j$ is a football for each $j$. 
Again, since the singular set of $Y_i$ consists of a union of embedded circles, the work of 
Dunbar in \cite{Dun} shows that $Y_i=\s^3/G_i$ for a finite subgroup $G_i$ of $SO(4)$ 
which preserves a Hopf fibration. It follows that $X_i$ is the mapping torus of a periodic 
diffeomorphism $f_i:\hat{Y}_i\rightarrow \hat{Y}_i$, where $\hat{Y}_i$ has a Seifert fibration
induced from the Hopf fibration and $f_i$ preserves the Seifert fibration on $\hat{Y}_i$.
By the assumption (iii), there is a $Y_i$ such that $\pi_1(\hat{Y}_i)$ is non-abelian.
With the following lemma (Lemma 6.4), we finish the proof by observing that $\pi_1(\hat{Y}_i)$ 
is completely determined by $\pi_1(X_i)$ which depends only on $G$.

\end{proof}

\begin{lemma}
Let $\hat{Y}$ be an elliptic $3$-manifold with non-abelian fundamental group, and let $\pi: \hat{Y}
\rightarrow B$ be the unique Seifert fibration on $\hat{Y}$. Suppose $f:\hat{Y}\rightarrow \hat{Y}$ 
is an orientation-preserving periodic diffeomorphism which preserves $\pi$. Then the multiplicity 
of any singular circle of the $3$-orbifold $Y=\hat{Y}/\langle f\rangle$ is bounded by a constant 
depending only on the multiplicities of the singular points of $B$.
\end{lemma}

\begin{proof} 
For any singular circle $\gamma$ in $Y$, the multiplicity of $\gamma$ equals the order of its 
isotropy subgroup. Let $f_\gamma$ be a generator of the isotropy subgroup, which is given 
by $f^k$ for some $k$. Since $f:\hat{Y}\rightarrow \hat{Y}$ preserves $\pi:\hat{Y}\rightarrow B$, 
so does $f_\gamma$, and there is an induced periodic diffeomorphism $\bar{f}_\gamma:
B\rightarrow B$ of the $2$-orbifold $B$. 

Since $\pi_1(\hat{Y})$ is non-abelian, $B$ is a turnover with multiplicities $(2,2,n)$, $(2,3,3)$,
$(2,3,4)$, or $(2,3,5)$. We shall discuss according to (i) $\bar{f}_\gamma$ is trivial, (ii) 
$\bar{f}_\gamma$ is non-trivial.

Suppose $\bar{f}_\gamma$ is trivial. Then $f_\gamma$ acts as a rotation on each fiber of 
$\pi:\hat{Y}\rightarrow B$. It follows easily that $\gamma$ must be an exceptional fiber of
$\pi$, and the order of $f_\gamma$ is a divisor of the multiplicity of the singular point 
$\pi(\gamma)\in B$. 

Suppose $\bar{f}_\gamma$ is non-trivial. Then there are two possibilities: (a) $\bar{f}_\gamma$ 
is orientation-preserving, (b) $\bar{f}_\gamma$ is orientation-reversing. In case (a), the order
of $\bar{f}_\gamma$ is either $2$ or $3$, and $\bar{f}_\gamma$ has two isolated fixed-points.
Moreover, $\gamma$ must be the fiber over one of the fixed-points of $\bar{f}_\gamma$.
It follows easily that the multiplicity of $\gamma$ equals the order of $\bar{f}_\gamma$, which
is at most $3$. In case (b), $\bar{f}_\gamma$ must be a reflection over a great circle in $B$ because
$\bar{f}_\gamma$ has a nonempty fixed-point (which contains $\pi(\gamma)$, for instance). Since
$f$ is orientation-preserving, $f_\gamma$ must be a reflection on the fibers over the great circle
fixed under $\bar{f}_\gamma$. It follows that the multiplicity of $\gamma$ equals $2$ in this case.

\end{proof}

\noindent{\bf Proof of Theorem 1.2}

\vspace{3mm}

Theorem 1.2 follows from Theorem 6.3 except in the following cases:

\begin{itemize}
\item [{(a)}] $\text{rank }z(G)>1$;
\item [{(b)}] $G$ is double-ended;
\item [{(c)}] $G$ is isomorphic to the $\pi_1$ of a Klein bottle;
\item [{(d)}] none of the above is true, and moreover, every vertex subgroup of a canonical JSJ decomposition of $G$ is an HNN extension of a finite cyclic group.
\end{itemize}

Cases (a), (c) are settled with the help of Theorems 4.3 and 6.2. Case (b) is settled by
Theorem 1.4, Lemma 4.2, and Theorem 1.6. (Note that in case (b) where $G$ is double-ended, 
we appeal to Theorem 1.6(2), where we observe that when $X$ is a mapping torus of a periodic 
diffeomorphism of a lens space, the number of possible lens spaces is bounded by a constant depending only on $\pi_1(X)$.)

For case (d), we shall continue with the proof of Theorem 6.3, where we are left with the 
situation that $\pi_1(\hat{Y_i})$ is finite cyclic for each $i$ and $\Sigma_j$ is a football for 
each $j$. Recall that $Y_i=\hat{Y}_i/\langle f_i\rangle$ for some periodic diffeomorphism 
$f_i$. Moreover, there is a Seifert fibration $pr_i: \hat{Y}_i\rightarrow B_i$ which is induced 
from a Hopf fibration and is preserved under $f_i$. 

We shall analyze the multiplicities of the singular circles in $Y_i$. To this end, let $\gamma$
be a singular circle and $f_\gamma$ be a generator of its isotropy subgroup. Denote by
$\bar{f}_\gamma: B_i\rightarrow B_i$ the induced map. If $\bar{f}_\gamma$ is orientation-reversing,
then as we showed in the proof of Lemma 6.4, the multiplicity of $\gamma$ is $2$. If
$\bar{f}_\gamma$ is orientation-preserving and switches the two singular points of $B_i$,
then the multiplicity of $\gamma$ is also $2$, as we argued in the proof of Lemma 6.4.
In the remaining cases where $\bar{f}_\gamma$ is either trivial or fixes the two singular points 
of $B_i$, or $B_i$ has no singular points at all, the multiplicity of $\gamma$ may not be bounded 
by a constant depending only on $G$, and we need to deal with it differently. 

Note that in either of the remaining cases, $Y_i=\s^3/G_i$ for a finite subgroup $G_i$ 
of $SO(4)$, which is given by
$$
\lambda \cdot (z_1,z_2)=(\lambda^{p_i}z_1,\lambda^{q_i}z_2), 
$$
where $\lambda=\exp(2\pi i/n_i)$ is a $n_i$-th root of unity and $\gcd(n_i,p_i,q_i)=1$. 
Set $u_i=\gcd(n_i,p_i)$, $v_i=\gcd(n_i,q_i)$. Then $Y_i$ has at most two singular circles 
of multiplicities $u_i,v_i$. Furthermore, if $\Sigma_j$ intersects the singular circle of 
multiplicity $u_i$ (resp. $v_i$), the index of $\pi_1(N_j)$ in $\pi_1(X_i)$ is $n_i/u_i$ 
(resp. $n_i/v_i$). Consequently, if both singular circles of $Y_i$ are intersected by $\Sigma_j$ 
for some $j$, then $u_i\leq n_i/v_i$, $v_i\leq  n_i/u_i$ are both bounded by a constant depending 
only on $G$ (cf. Proposition 3.5). Clearly, we are done for (iii) in the proof of Theorem 6.3 
if there exists a $Y_i$ for which such a situation occurs. 

We are left to examine the case where for each $i$, there is exactly one singular circle of 
$Y_i$ which is intersected by $\Sigma_j$ for some $j$. In this case, we shall apply Lemma 6.1
and change the Seifert-type $\s^1$-fibrations $\pi_i:X_i\rightarrow Y_i$ in the fiber-sum decomposition
of $X$ to $\pi^\prime_i:X_i\rightarrow Y_i^\prime$. Note that with the new fibrations $\pi_i^\prime$,
each $N_j$ is fibered over an ordinary $2$-sphere. Consequently, up to suitable modifications
of the Seifert-type $\s^1$-fibrations, the number of objects in (iii) in the proof of Theorem 6.3 is bounded 
by a constant depending only on $G$, from which Theorem 1.2 follows. 

\vspace{5mm}

\noindent{\bf Acknowledgments}: The author is in debt to Slawomir Kwasik for many helpful
conversations throughout the course of the work. This work was partly done during 
the author's visit at the Institute for Advanced Study, and he wishes to thank the institute for its 
hospitality and financial support. Finally, the author thanks the referee for a very careful reading 
of the article and for the many suggestions which have improved considerably the exposition of
the paper.

\vspace{2mm}

{\Small University of Massachusetts, Amherst.\\
{\it E-mail:} wchen@math.umass.edu

\end{document}